\documentclass[a4paper, 11pt, twoside, openany]{article}
\usepackage{amsmath,amssymb,stmaryrd,mathtools,amsthm}
\usepackage[all]{xy}
\usepackage{fancyhdr}
\usepackage{footmisc}
\usepackage{hyperref}
\usepackage[margin=22mm]{geometry}
\usepackage[UKenglish]{datetime}
\usepackage{graphicx}
\usepackage{wrapfig}
\usepackage[usenames, dvipsnames]{color}
\usepackage{accents}
\usepackage{cases}
\usepackage{enumerate}
\usepackage{tikz-cd}
\usepackage{multirow}
\usepackage[normalem]{ulem}
\usepackage{algorithm}
\usepackage{algorithmic}

\newcommand{\pp}[2]{\frac{\partial #1}{\partial #2}} 
\newcommand{\dd}[2]{\frac{\delta #1}{\delta #2}}

\newcommand{\bs}[1]{\boldsymbol{#1}}

\newtheorem{definition}{Definition}[section]
\newtheorem{proposition}{Proposition}[section]
\theoremstyle{definition}
\newtheorem{remark}{Remark}[section]

\numberwithin{equation}{section}

\newcounter{savefootnote}
\newcounter{symfootnote}
\newcommand{\symfootnote}[1]{%
   \setcounter{savefootnote}{\value{footnote}}%
   \setcounter{footnote}{\value{symfootnote}}%
   \ifnum\value{footnote}>8\setcounter{footnote}{0}\fi%
   \let\oldthefootnote=\thefootnote%
   \renewcommand{\thefootnote}{\fnsymbol{footnote}}%
   \footnote{#1}%
   \let\thefootnote=\oldthefootnote%
   \setcounter{symfootnote}{\value{footnote}}%
   \setcounter{footnote}{\value{savefootnote}}%
}

\begin{document}
%
%
%
\begin{center}
{\Large A structure-preserving discontinuous Galerkin scheme for the Cahn-Hilliard equation including time adaptivity}
\end{center}
\vspace{-3mm}
\hrulefill
\begin{center}
{Golo A. Wimmer$^{1}$\symfootnote{correspondence to: gwimmer@lanl.gov}, Ben S. Southworth$^1$, Qi Tang$^{2, 1}$}\\
\vspace{2mm}
{\textit{$^1$Theoretical Division, Los Alamos National Laboratory, Los Alamos, New Mexico 87545 USA \\[.5em]
 $^2$School of Computational Science and Engineering, Georgia Institute of Technology, Atlanta, Georgia 30332 USA}}\\
\vspace{4mm}
\today
\end{center}


\begin{abstract}
We present a novel spatial discretization for the Cahn-Hilliard equation including transport. The method is given by a mixed discretization for the two elliptic operators, with the phase field and chemical potential discretized in discontinuous Galerkin spaces, and two auxiliary flux variables discretized in a divergence-conforming space. This allows for the use of an upwind-stabilized discretization for the transport term, while still ensuring a consistent treatment of structural properties including mass conservation and energy dissipation. Further, we couple the novel spatial discretization to an adaptive time stepping method in view of the Cahn-Hilliard equation's distinct slow and fast time scale dynamics. The resulting implicit stages are solved with a robust preconditioning strategy, which is derived for our novel spatial discretization based on an existing one for continuous Galerkin based discretizations. Our overall scheme's accuracy, robustness, efficient time adaptivity as well as structure preservation and stability with respect to advection dominated scenarios are demonstrated in a series of numerical tests.
\end{abstract}
\section{Introduction}
The Cahn-Hilliard (CH) equation models phase separation using a continuous phase $\phi$ \cite{cahn1958free}, and can be used to model a variety of physical phenomena including spinodal decomposition, material microstructure evolution \cite{chen2002phase, kim2016basic}, and multiphase fluid flows \cite{anderson1998diffuse}. One common form is given by
\begin{equation}
\pp{\phi}{t} - \nabla \cdot \left(d(\phi) \nabla \big(\phi^3 - \phi - \epsilon^2 \Delta \phi \big)\right) = 0, \label{CH_cts_simple}
\end{equation}
for diffusion coefficient $d$ and $\epsilon > 0$. Numerical discretizations of the CH equation are required to tackle a series of challenges including a highly stiff phase field evolution due to the occurrence of a biharmonic operator as well as nonlinear terms. As an additional challenge to numerical solver efficiency, the latter nonlinearity can give rise to both very fast and slow time scales. Further, the equation contains structural properties including a conserved total amount of phase and dissipated energy, up to inflow/outflow effects and transfers to and from kinetic energy in the context of a CH-fluid flow coupled system of equations. This large variety of application areas as well as modeling challenges \cite{miranville2019cahn} associated with the equation has given rise to an equally rich literature of numerical discretization methods for the CH equations (see e.g., the review \cite{lee2014physical}), including coupled with flow (e.g., \cite{dehghan2021numerical, diegel2015analysis, khanwale2022fully, kim2004conservative} and references therein).

In the context of fluid flow, one popular way to discretize the CH equation is given by discontinuous Galerkin (DG) finite element methods (e.g. \cite{fu2020divergence} and references therein), which allow for easy meshing of general domains, are readily available in finite element libraries (e.g., \cite{anderson2021mfem, FiredrakeUserManual}), and in particular -- unlike continuous Galerkin methods (CG) -- ensure local mass conservation. The equation is typically posed in a mixed form, introducing the chemical potential
\begin{equation}
\mu = \phi^3 - \phi - \epsilon \Delta \phi \label{mu_simple}
\end{equation}
as an auxiliary variable. This results in two separate Laplacian operators, which are then discretized weakly in a non-conforming fashion, requiring some form of flux-type contribution defined on mesh facets, including stabilization terms to ensure coercivity (e.g., \cite{feng2016analysis, kay2009discontinuous, kirk2023numerical, liu2019numerical, liu2021unconditionally} and references therein). However, such facet contributions then no longer allow for a natural energy dissipation law, and require special reformulations of the energy functional (such as \cite{liu2021unconditionally, zhao2023numerical}) in order to recover a discrete analogue. Next to these considerations regarding spatial discretizations, specialized time discretizations are also required, and are aimed at adaptive time stepping (e.g. \cite{fu2021linear} for a CH/flow coupled problem), provable energy stability (e.g., \cite{akrivis2019energy, chen2019positivity, diegel2016stability, wu2014stabilized}), and general long-term stability such as through the scalar auxiliary variable (SAV) approach developed for gradient flows including the Cahn-Hilliard equation \cite{shen2018scalar, shen2019new}. Other recent advances in stable time discretization methods for the CH equation include exponential time integrators \cite{fu2022energy, fu2024higher} and implicit-explicit Runge--Kutta integrators \cite{fu2024energy}.

In this work, we present a novel discontinuous Galerkin based spatial discretization that naturally preserves the CH equation's energy dissipative property. Rather than relying on non-conforming type methods including facet integrals, we discretize the two Laplacians occurring in the CH equation's phase field-chemical potential mixed form using a mixed finite element method \cite{arnold2006finite, boffi2013mixed}. In particular, the phase field's and chemical potential's gradients are represented weakly using divergence-conforming auxiliary variables, leading to an overall formulation involving a total of four variables. We show the four variable formulation to be the dual form of a standard conforming mixed CG-based discretization \cite{diegel2015analysis} along a discrete de-Rham complex of compatible finite element spaces. Based on this observation, we further present a robust preconditioning strategy based on swapping rows and applying a diagonal preconditioner as devised in \cite{brenner2018robust} for the standard CG-based formulation. To the best of our knowledge, this forms the first penalty parameter-free DG spatial discretization method for the CH equations satisfying the energy dissipative property, using a strongly consistent discrete form of the energy.

In order to test our new spatial discretization on a variety of different scenarios for the CH equation, we include a flow term discretized using standard DG upwinding, and further consider adaptive time stepping using the TR-BDF2 method \cite{bank1985transient}. In particular, this leads to a fully implicit scheme, necessitating the aforementioned specially designed preconditioning strategy. Our test scenarios include verification for our novel spatial discretization's order of accuracy in accordance with expected values for DG-based mixed discretizations of the Laplacian, as well as structure-preserving properties for mass conservation and energy dissipation. Additionally, we verify our preconditioning method's robustness. Finally, we consider fast/slow time scale dynamics using time steps that vary by several orders of magnitude, and examine our novel method's behavior in advection dominated regimes.

The remainder of this work is structured as follows: In Section \ref{sec_background}, we review the Cahn-Hilliard equation and its structural properties, as well as the additional terms included when coupling to fluid equations. Further, we include a brief description of a standard, CG-based mixed discretization. In Section \ref{sec_novelty}, we introduce our novel scheme, discussing a) the compatible finite element based spatial discretization and its structure-preserving properties, b) the adaptive time discretization, and c) the preconditioning strategy for the resulting implicit systems of equations. In Section \ref{sec_Numerical_results}, we present and discuss numerical results. Finally, in Section \ref{sec_conclusion}, we review our results and discuss possible future work.
%
\section{Background} \label{sec_background}
In this section, we review the Cahn-Hilliard equation and its structural properties, and briefly describe its coupling to fluid equations. Further, we outline a standard, CG-based spatial discretization thereof, which preserves its structure.
\subsection{Cahn-Hilliard equation}
A more general formulation of the CH equation \eqref{CH_cts_simple}, for $\phi$ defined over a domain $\Omega$, is given by \cite{kim2016basic}
\begin{equation}
\pp{\phi}{t} - \nabla \cdot \left(d(\phi) \nabla \big(F'(\phi) - \epsilon^2 \Delta \phi \big)\right) = S, \label{CH_cts}
\end{equation}
for potential $F$ and forcing term $S$. Values of $\phi = -1$, $\phi = 1$ indicate the two phases, respectively, and values inbetween indicate a transition zone. As in \eqref{mu_simple}, the expression to which the gradient is applied can be identified as the chemical potential
\begin{equation}
\mu = F'(\phi) - \epsilon^2 \Delta \phi; \hspace{1cm} \pp{\phi}{t} - \nabla \cdot \left(d(\phi) \nabla \mu\right) = S. \label{CH_mixed_cts}
\end{equation}
The non-negative diffusion coefficient $d$ can take various forms, including the transmission type $d(\phi) = d_0 (1 - \phi)(1 + \phi)$ for $d_0 > 0$, as well as a simple constant $d_0$. In this work, for simplicity, we consider $d$ as a constant, noting that the discrete structural properties to be derived for our novel spatial discretization in Section \ref{sec_space} below equally hold for non-constant $d$. Additionally, as in \eqref{CH_cts_simple} we will consider a double-well potential of the form
\begin{equation}
F(\phi) = \frac{1}{4}\left(1 - \phi^2 \right)^2, \hspace{1cm} \text{so that} \hspace{1cm} F'(\phi) = \pp{F}{\phi} = \phi^3 - \phi. \label{F_Fp}
\end{equation}
Similarly, the results of Section \ref{sec_space} hold true for general types of $F$. Finally, in this work, we consider \eqref{CH_cts} either on periodic domains or equipped with homogeneous Neumann boundary conditions of the form
\begin{equation}
\mathbf{n} \cdot \nabla \phi = 0, \hspace{1cm} \mathbf{n} \cdot \nabla \mu = 0, \label{Neumann_BCs}
\end{equation}
for outward unit normal vector $\mathbf{n}$ at the domain's boundary $\partial \Omega$.

\textbf{Structural properties.}
The Cahn-Hilliard equation \eqref{CH_cts} satisfies total mass conservation up to the forcing term $S$ \cite{miranville2019cahn}; this can readily be seen since the equation is of flux form
\begin{equation}
\pp{\phi}{t} - \nabla \cdot \mathbf{j} = S,
\end{equation}
for flux $\mathbf{j} = d(\phi)\nabla \mu$. Integrating over $\Omega$ and applying the divergence theorem together with the homogeneous Neumann boundary condition for $\mu$, we find
\begin{equation}
\frac{d}{dt}\int_\Omega \phi \; dx = \int_\Omega \pp{\phi}{t} \; dx = \int_\Omega \nabla \cdot \mathbf{j} \; dx + \int_\Omega S \; dx = \int_{\partial \Omega} d(\phi) \mathbf{n} \cdot \nabla \mu \; dx + \int_\Omega S \; dx = \int_\Omega S \; dx. \label{mass_dt}
\end{equation}
Additionally, the Cahn-Hilliard equation admits a non-increasing energy functional, again up to the forcing term. The latter functional is given by
\begin{equation}
H_{C\!H}(\phi) = \int_\Omega \left(F(\phi) - \frac{\epsilon^2}{2}|\nabla \phi|^2\right)dx, \label{CH_energy}
\end{equation}
and in view of its rate of change in time, we first derive its functional derivative with respect to $\phi$. For a suitable function space $\mathbb{V}_\phi$ in which $\phi$ is defined, the functional derivative evaluated at $\phi$ is given by
\begin{align}
\left\langle \eta, \dd{H_{C\!H}}{\phi}\right\rangle \coloneqq \lim_{\alpha \rightarrow 0} \frac{1}{\alpha}\left(H_{C\!H}(\phi + \alpha \eta) - H_{C\!H}(\phi)\right) = \left\langle \eta, \dd{F}{\phi} \right\rangle + \left\langle \nabla \eta, \epsilon^2 \nabla \phi \right\rangle && \forall \eta \in \mathbb{V}_\phi, \label{var_form_H}
\end{align}
where brackets denote the $L^2$-inner product in $\Omega$. In a strong form with the above homogenous Neumann boundary condition for $\phi$, we then get the familiar expression
\begin{equation}
\dd{H_{C\!H}}{\phi} = \phi^3 - \phi - \epsilon^2 \Delta \phi = \mu,
\end{equation}
after integrating by parts and using $F$ as in \eqref{F_Fp}, and further noting that \eqref{var_form_H} holds true for any test function $\eta \in \mathbb{V}_\phi$. Using this, the non-increasing total energy property (up to forcing $S$) can then be derived using the chain rule in time and integration by parts in space according to
\begingroup
\addtolength{\jot}{2mm}
\begin{align}
\frac{dH_{C\!H}}{dt} = &\left\langle \dd{H_{C\!H}}{\phi}, \pp{\phi}{t} \right\rangle=\left\langle \mu, \pp{\phi}{t} \right\rangle= \left\langle \mu, \nabla \cdot \left(d(\phi) \nabla \mu\right)\right\rangle + \left\langle \mu, S \right\rangle \nonumber \\
= & -\left\langle \nabla \mu, d(\phi) \nabla \mu\right\rangle + \left\langle \mu, S \right\rangle = - \left\|\sqrt{d(\phi)}\nabla \mu\right\|_2^2 + \left\langle \mu, S \right\rangle. \label{HCH_dt}
\end{align}
\endgroup
In particular, in the absence of forcing $S$ and for homogeneous Neumann boundary conditions, we find that $\tfrac{dH}{dt} \le 0$.

\textbf{Coupling to fluid equations}. In the presence of fluid flow, the phase field is advected using a transport operator in flux form, that is
\begin{equation}
\pp{\phi}{t} - \nabla \cdot \left(d(\phi) \nabla \big(F'(\phi) - \epsilon^2 \Delta \phi \big)\right) + \nabla \cdot (\mathbf{u} \phi) = S, \label{CH_cts_flow}
\end{equation}
where $\mathbf{u}$ denotes the fluid flow field. In this work, for a simpler presentation we assume free-slip boundary conditions of the form
\begin{equation}
\mathbf{u} \cdot \mathbf{n} = 0 \;\; \text{for} \;\; \mathbf{x} \in \partial \Omega,
\end{equation}
noting that what follows applies to other types of boundary conditions in an analogous fashion. Assuming for simplicity the fluid system's density and thermodynamic variables to be independent of $\phi$, the corresponding coupling then consists of an additional forcing term $\tfrac{\phi}{\rho}\nabla \mu$ in the momentum equation
\begin{equation}
\pp{\mathbf{u}}{t} + \dots + \frac{\phi}{\rho}\nabla \mu = 0, \label{u_eqn_cts}
\end{equation}
for fluid density $\rho$ and dots denoting additional terms such as velocity transport and the pressure gradient term. As for the Cahn-Hilliard equation, we can explore structural properties by considering the coupling's energetics. This time, the total energy functional consists of \eqref{CH_energy}, together with the kinetic energy
\begin{equation}
H_{K\!E}(\rho, \mathbf{u}) = \frac{1}{2} \int_\Omega \rho |\mathbf{u}|^2 \; dx,
\end{equation}
and possibly other types of sub-energy, such as a total internal energy. We can then show the phase transport term in \eqref{CH_cts_flow} and the corresponding coupling term in \eqref{u_eqn_cts} to be energetically consistent. Noting that it can be shown that $\dd{H_{K\!E}}{\mathbf{u}} = \rho \mathbf{u}$, we find using the chain rule in time,
\begingroup
\addtolength{\jot}{2mm}
\begin{subequations} \label{H_consistent_cts}
\begin{align}
& \frac{dH_{K\!E}}{dt} = \left\langle \dd{H_{KE}}{\rho}, \pp{\rho}{t} \right\rangle + \left\langle \dd{H_{KE}}{\mathbf{u}}, \pp{\mathbf{u}}{t} \right\rangle =  \left\langle \dd{H_{KE}}{\rho}, \pp{\rho}{t} \right\rangle - \dots - \left\langle \mathbf{u}\phi, \nabla \mu \right\rangle, \label{HKE_dt_flow} \\
&\frac{dH_{C\!H}}{dt} =  \left\langle \dd{H_{C\!H}}{\phi}, \pp{\phi}{t} \right\rangle = - \left\|\sqrt{d(\phi)}\mu\right\|_2^2 + \left\langle \mu, S \right\rangle - \left\langle \mu, \nabla \cdot (\mathbf{u} \phi) \right\rangle \nonumber \\
&\hspace{35mm} =  - \left\|\sqrt{d(\phi)}\mu\right\|_2^2 + \left\langle \mu, S \right\rangle + \left\langle \nabla \mu, \mathbf{u} \phi\right\rangle \label{HCH_dt_flow}
\end{align}
\end{subequations}
\endgroup
where in \eqref{HKE_dt_flow} we substituted equation \eqref{u_eqn_cts} and cancelled $\rho$ in the right-hand side's last term. Further, in \eqref{HCH_dt_flow}, we substituted equation \eqref{CH_cts_flow} and inserted the derivation \eqref{HCH_dt} for the first two terms on the right-hand side. Finally, we applied integration by parts in space for the last term on the right-hand side, using the free-slip boundary conditions for $\mathbf{u}$. In particular, energetic consistency for the velocity-phase field coupling then follows, since the transport and forcing terms $\nabla \cdot (\mathbf{u}\phi)$ and $\tfrac{\phi}{\rho}\nabla \mu$, respectively, yield a consistent energy exchange of the form
\begin{equation}
\frac{dH_{K\!E}}{dt} = \dots - \left\langle \mathbf{u}\phi, \nabla \mu \right\rangle, \hspace{1cm} \frac{dH_{C\!H}}{dt} =  \dots  + \left\langle \mathbf{u} \phi, \nabla \mu\right\rangle. \label{KE_CH_energy}
\end{equation}

%
\subsection{Standard CG discretization} \label{background_CG}
The Cahn-Hilliard equation \eqref{CH_cts} can be discretized in a structure-preserving manner in space using a mixed continuous Galerkin (CG) scheme \cite{diegel2016stability}, where $\mu$ is considered as an auxiliary variable. For a suitable CG space $\mathbb{V}_{\text{CG}}$ -- to be specified further in the numerical results section -- we find $(\phi_h, \mu_h) \in \mathbb{V}_{\text{CG}} \times \mathbb{V}_{\text{CG}}$
\begingroup
\addtolength{\jot}{2mm}
\begin{subequations} \label{CG_discr}
\begin{align}
&\left\langle \eta, \pp{\phi_h}{t} \right\rangle + \left\langle \nabla \eta, d(\phi_h) \nabla \mu_h \right\rangle = \langle \eta, S \rangle, & \forall \eta \in \mathbb{V}_{\text{CG}}, \label{CG_discr_phi}\\
&\left\langle \eta, \mu_h \right\rangle = \left\langle \eta, \phi_h^3 - \phi_h \right\rangle + \left\langle \nabla \eta, \epsilon^2 \nabla \phi_h \right\rangle , & \forall \eta \in \mathbb{V}_{\text{CG}},
\end{align}
\end{subequations}
\endgroup
where we have incorporated the homogeneous Neumann boundary conditions weakly. The structural properties -- given by the mass conservation \eqref{mass_dt} and non-increasing energy \eqref{HCH_dt} -- can be derived in an analogous way to the continuous case. For instance, for the energy development we consider a discrete energy $H_{CG}(\phi_h)$ defined as \eqref{CH_energy} with $\phi$ replaced by $\phi_h$, with discrete variational derivative $\dd{H_{C\!H}}{\phi_h} \in \mathbb{V}_{\text{CG}}$ derived according to
\begin{align}
\left\langle \chi, \dd{H_{C\!H}}{\phi_h}\right\rangle \!=\! \left\langle \chi, \dd{F}{\phi_h} \right\rangle \!+\! \left\langle \nabla \chi, \epsilon^2 \nabla \phi_h \right\rangle \!=\! \left\langle \chi, \phi_h^3 - \phi_h \right\rangle \!+\! \left\langle \nabla \chi, \epsilon^2 \nabla \phi_h \right\rangle \!=\! \langle \chi, \mu_h \rangle && \forall \chi \in \mathbb{V}_{\text{CG}}. \label{var_form_H_discr}
\end{align}
In particular, we can then use \eqref{CG_discr_phi} with $\eta = \mu_h$ to obtain a discrete version of \eqref{HCH_dt} given by
\begin{equation}
\frac{dH_{C\!H}}{dt} = \left\langle \dd{H_{C\!H}}{\phi_h}, \pp{\phi_h}{t} \right\rangle = \left\langle \mu_h, \pp{\phi_h}{t} \right\rangle = - \left\|\sqrt{d(\phi_h)}\nabla \mu_h\right\|_2^2 + \langle \mu_h, S \rangle.
\end{equation}

Finally, the standard CG discretization readily extends to the case including fluid flow, discretizing the transport and forcing terms according to
\begingroup
\addtolength{\jot}{2mm}
\begin{subequations} \label{CG_coupling}
\begin{align}
&\left\langle \eta, \pp{\phi_h}{t} \right\rangle + \left\langle \nabla \eta, d(\phi_h) \nabla \mu_h \right\rangle  - \left\langle \nabla \eta, \mathbf{u}_h \phi_h \right\rangle = \langle \eta, S \rangle, & \forall \eta \in \mathbb{V}_{\text{CG}}, \label{CG_CH_flow} \\
&\left\langle \rho_h \mathbf{w}, \pp{\mathbf{u}_h}{t} \right\rangle + \left\langle \nabla \mu_h, \mathbf{w} \phi_h \right\rangle= 0 & \forall \mathbf{w} \in \mathbb{V}_{u_h},
\end{align}
\end{subequations}
\endgroup
for discrete fields $\mathbf{u}_h$, $\rho_h$ discretized in suitable function spaces. Again, energetic consistency then follows analogously to the continuous case \eqref{H_consistent_cts}, by setting $\eta = \mu_h$ and $\mathbf{w} = \mathbf{u}_h$ in the discrete phase field and momentum equations, respectively.

We conclude the background section with some additional observations on the CG based discretization \eqref{CG_discr}, which serve as a guide for our novel discontinuous Galerkin (DG) based scheme to be introduced in the next section.


\textbf{Efficient and scalable solvers.} Due to the severe time step restriction arising from the fourth order operator within the Cahn-Hilliard equation, time discretizations for the stiff linear part are required to be implicit. This results in a nonlinear coupled system of equations based on \eqref{CG_discr}, which needs to be solved in each time step (or each implicit stage). One possible way to achieve this in an efficient and scalable manner is described in \cite{brenner2018robust}. In the latter work, a Newton method is applied to the nonlinear system, where each iteration requires solving the Jacobian
\begin{equation}
\begin{bmatrix}
M_{CG} & \sqrt{\delta t} L \\
-\sqrt{\delta t}\left(F'' + \epsilon^2L\right) & M_{CG}
\end{bmatrix}
\begin{bmatrix}
\delta \phi\\
\delta \mu\\
\end{bmatrix} 
=
\begin{bmatrix}
R_\phi \\
R_\mu
\end{bmatrix},
\label{phimu_system}
\end{equation}
for time step $\delta t$ and where $(\delta \phi, \delta \mu)$ and $(R_\phi, R_\mu)$ denote incremental updates and residuals, respectively. The block matrices $M_{CG}$ and $L$ denote the CG mass and stiffness matrices, respectively. Additionally, given their choice of implicit time discretization, $F''$ is defined as $F''_{ij} = \langle3(\phi_h^k)^2 \eta_i, \eta_j\rangle$ for $\eta_i, \eta_j \in \mathbb{V}_{\text{CG}}$, where $\phi_h^k$ denotes the solution at the $k^{th}$ Newton iteration. The diffusion coefficient is set to $d = 1$, and for a simpler presentation we skipped modifications in \cite{brenner2018robust} related to the mean phase $\int_\Omega \phi_h \; dx$.

In \cite{brenner2018robust}, to solve for \eqref{phimu_system}, the row-swapped reformulation
\begin{equation}
\begin{bmatrix}
-\sqrt{\delta t}\left(F'' + \epsilon^2L\right) & M_{CG} \\
M_{CG} & \sqrt{\delta t} L
\end{bmatrix}
\begin{bmatrix}
\delta \phi\\
\delta \mu\\
\end{bmatrix} 
=
\begin{bmatrix}
R_\mu \\
R_\phi
\end{bmatrix},
\label{phimu_system_swapped}
\end{equation}
is considered instead, which is preconditioned using
\begin{equation}
\begin{bmatrix}
M_{CG} + \sqrt{\delta t} \epsilon^2 L & 0 \\
0 & M_{CG} + \sqrt{\delta t} L
\end{bmatrix}. \label{CG_diag_precon}
\end{equation}
The preconditioner is shown to be robust when applied to an approximated version of \eqref{phimu_system_swapped}, where $F''$ is replaced by $3M_{DG}$. The latter approximation is justified by assuming that $\phi_h^k$ is close to either $1$ or $-1$ in most of the domain, so that $\left(\phi_h^{k}\right)^2 \approx 1$. This preconditioner in \cite{brenner2018robust} was numerically demonstrated to be scalable with respect to $\epsilon$ and mesh size for a given time step (up to the approximation used to replace $F''$), as well as with respect to mesh size and time step for a given $\epsilon$. Finally, the two diagonal blocks in \eqref{CG_diag_precon} correspond to a weak CG-based discretization of the Helmholtz operator (with positive coefficient), and can be solved for efficiently using classical AMG \cite{benzi2005numerical,ruge1987algebraic}.

\textbf{Transport.}
While the mixed CG based spatial discretization \eqref{CG_discr} is structure-preserving and can be solved efficiently and accurately in the context of implicit time discretizations, its accuracy may degrade when coupled to fluid flow as in \eqref{CG_coupling}, in the case of strongly advection dominated regimes. This holds true since the transport discretization does not contain any measures for stabilization, and the higher  P\'eclet number -- that is the ratio of advective to diffusive transport rate -- the more transport noise we should expect \cite{kuzmin2010guide}. In practice, for the CH equations, the noise is then either ``smeared out'' by the biharmonic operator, or triggers phase separation in an otherwise phase neutral zone of value $\phi = 0$, both of which lead to a deteriorated overall solution accuracy.

Unfortunately, a straightforward application of standard stabilization methods such as SUPG, or more generally variational multiscale methods \cite{hughes2010stabilized}, alter the time derivative and chemical potential flux terms, and therefore render the discretization no longer structure-preserving, which is known to be critical to obtain a physical solution \cite{shen2019new}. Further, regardless of stability, one may prefer a local form of mass conservation as provided for by DG methods, unlike CG methods where its mass is conserved only in a global sense. However, as stated in the introduction, current penalty based DG methods in the literature -- such as \cite{kay2009discontinuous} for a prototypical form of an interior penalty based setup -- also do not naturally preserve the equation's energetic structure, due to the introduction of facet integrals and penalty terms.

%
%
\section{Discretization} \label{sec_novelty}
\textbf{Motivation.} In order to avoid the DG facet integrals and penalty terms mentioned at the end of the previous section, we first observe that the two discrete, CG-based equations \eqref{CG_discr} for $\phi_h$ and $\mu_h$ contain discrete weak formulations for the Poisson equation
\begin{equation}
\Delta v = -f, \hspace{1cm} \mathbf{n} \cdot \nabla v = 0 \;\;\text{for}\;\; \mathbf{x} \in \partial \Omega, \label{Poisson}
\end{equation}
for unknown scalar field $v$, and known right-hand side $f$. In particular, the preconditioning strategy \eqref{CG_diag_precon} by \cite{brenner2018robust} outlined above is centered on solving two such equations in a decoupled fashion.

Alternatively, the Poisson equation can be posed in a mixed form using $u \in L^2(\Omega)$, and its gradient $\bs{\chi}$ as an auxiliary variable. For this purpose, we consider so-called compatible finite element spaces, which follow the de-Rham complex
\begin{displaymath}
    \xymatrix{
        H^1(\Omega) \ar[r]^{\nabla} \ar[d]^{\pi^0} & H(\text{{\normalfont curl}};\Omega) \ar[r]^{\nabla \times} \ar[d]^{\pi^1} & H(\text{{\normalfont div}};\Omega)\ar[d]^{\pi^2} \ar[r]^{\nabla \cdot} & L^2(\Omega) \ar[d]^{\pi^3} \\
      \mathbb{V}_{\text{CG}} \ar[r]^{\nabla}  &\mathbb{V}_{H(\text{curl})} \ar[r]^{\nabla \times} & \mathbb{V}_{H(\text{div})} \ar[r]^{\nabla \cdot}  & \mathbb{V}_{\text{DG}}}
\end{displaymath}
for bounded projections $\pi_i$, $i \in \{1, 2, 3, 4\}$. Further, $H^1(\Omega)$, $H(\text{{\normalfont curl}};\Omega)$ and $H(\text{{\normalfont div}};\Omega)$ denote the gradient-, curl- and divergence-conforming Sobolev spaces, respectively, and the spaces in the bottom row denote discrete versions thereof. In other words, the de-Rham complex and its discrete counterpart define a chain of spaces such that one space maps to another via a differential operator; for details, see \cite{arnold2006finite,boffi2013mixed}.  Given such compatible discrete spaces $\mathbb{V}_{\text{DG}}$ and $\mathbb{V}_{H(\text{div})}$ (to be specified further in the numerical results section), a stable mixed discretization is then given by finding $\bs{\chi}_h, v_h \in (\mathbb{V}_{H(\text{div})}, \mathbb{V}_{\text{DG}})$ such that \cite{duran2008mixed}
\begingroup
\begin{subequations} \label{Poisson_mixed_discr}
\begin{align}
&\langle \psi, \nabla \cdot \bs{\chi}_h \rangle = \langle \psi, -f \rangle &\forall \psi \in  \mathbb{V}_{\text{DG}},\\
&\langle \mathbf{w}, \bs{\chi}_h \rangle = - \langle \nabla \cdot \mathbf{w}, v_h \rangle &\forall \mathbf{w} \in \mathring{\mathbb{V}}_{H(\text{div})},
\end{align}
\end{subequations}
\endgroup
where $\mathring{\mathbb{V}}_{H(\text{div})}$ denotes the space $\mathbb{V}_{H(\text{div})}$ equipped with homogeneous boundary conditions, that is
\begin{equation}
\mathring{\mathbb{V}}_{H(\text{div})} \subset \mathring{H}(\text{div}; \Omega) = \{\mathbf{w} \in H(\text{div}; \Omega) \; | \; \mathbf{w} \cdot \mathbf{n}_{|_{\partial \Omega}} = 0\}.
\end{equation}
In other words, in this setting the Neumann boundary conditions for $v$ appear as Dirichlet boundary conditions for the discrete gradient $\bs{\chi}_h$ of $v$; this works since in practice, in order to ensure a well-defined divergence operator, discrete divergence conforming spaces contain vector functions with continuous normal components across facets. This is achieved through degrees of freedom defined as normal components of the vector field along element facets, including boundary facets. For further details on mixed discretizations for the Poisson equation, including stability considerations, see e.g., \cite{boffi2013mixed,duran2008mixed}. In particular, for our purposes, this is a conforming discretization that does not require facet integrals or penalty terms.

\subsection{Structure-preserving spatial discretization} \label{sec_space}

Given the motivation above, we next construct our novel spatial discretization for the Cahn-Hilliard equation \eqref{CH_mixed_cts} in a mixed form for $(\phi, \mu)$, discretizing the two Laplacians according to the mixed finite element method \eqref{Poisson_mixed_discr}. For this purpose, we proceed as in Section \ref{sec_background}, and first consider the case without fluid coupling.

\begin{definition} \label{def_dg}
Consider compatible finite element spaces $\mathbb{V}_{H(\text{div})}, \mathbb{V}_{\text{DG}}$. Further, set $\mathring{\mathbb{V}}_{H(\text{div})}$ to be $\mathbb{V}_{H(\text{div})}$ equipped with homogenous boundary conditions. Then the mixed DG-based discretization for the Cahn-Hilliard equation \eqref{CH_mixed_cts} with $d(\phi) > 0$ and equipped with homogeneous boundary conditions \eqref{Neumann_BCs} is given by finding $(\phi_h, \mathbf{j}_h, \mu_h, \bs{\sigma}_h) \in (\mathbb{V}_{\text{DG}}, \mathring{\mathbb{V}}_{H(\text{div})}, \mathbb{V}_{\text{DG}}, \mathring{\mathbb{V}}_{H(\text{div})})$, such that
\begingroup
\addtolength{\jot}{2mm}
\begin{subequations} \label{CH_DG}
\begin{align}
&\left\langle \psi, \pp{\phi_h}{t} + \nabla \cdot \mathbf{j}_h \right\rangle = \langle \psi, S\rangle &\forall \psi \in \mathbb{V}_{\text{DG}}, \label{CH_dg_phi} \\
&\left\langle \tfrac{1}{d(\phi_h)}\mathbf{v}, \mathbf{j}_h\right\rangle - \left\langle \nabla \cdot \mathbf{v}, \mu_h \right\rangle = 0 &\forall \mathbf{v} \in \mathring{\mathbb{V}}_{H(\text{div})}, \label{j_dg}\\
&\left\langle \psi, \mu_h - \phi_h^3 + \phi_h + \epsilon^2 \nabla \cdot \bs{\sigma}_h \right\rangle = 0 &\forall \psi \in \mathbb{V}_{\text{DG}}, \label{mu_dg}\\
&\langle \mathbf{v}, \bs{\sigma}_h \rangle + \langle \nabla \cdot \mathbf{v}, \phi_h \rangle = 0 &\forall \mathbf{v} \in \mathring{\mathbb{V}}_{H(\text{div})}. \label{sigma_dg}
\end{align}
\end{subequations}
\endgroup
\end{definition}
In other words, we introduced two auxiliary variables corresponding to the (negative) discrete phase field flux $\mathbf{j}_h$, as well as the discrete phase field's gradient $\bs{\sigma}_h$. As in the motivational example of the Poisson equation, the Neumann boundary conditions now appear in the form of Dirichlet boundary conditions for the two auxiliary variables.

Note that in the flux equation \eqref{j_dg} -- which corresponds to a weak discrete form of -$d(\phi)\nabla \mu$ -- we divided by $d(\phi)$ since otherwise, we would have a term $\nabla \cdot (d(\phi)\mathbf{v})$. This would require additional facet integral terms since a $\phi_h$-dependent diffusion coefficient need not be continuous across facets, for $\phi_h \in \mathbb{V}_{\text{DG}}$. The resulting necessary assumption $d(\phi) > 0$ for \eqref{j_dg} need not hold true for e.g., the transmission type coefficient $d(\phi) = d_0 (1 - \phi)(1 + \phi)$. In practice, we expect that this can be resolved simply by adding a small tolerance $\epsilon_0 \ll 1$ to $d$, that is $d(\phi_h) \rightarrow \text{max}\big(d(\phi_h), \epsilon_0\big)$. Note that the below proofs on structural properties hold true irrespective of the definition of $d$, including modifications thereof such as the aforementioned tolerance.
%
\begin{remark}
Since both Laplacian operators are discretized using a stable method, one can expect the overall method to be stable. In fact, for constant diffusion coefficient $d(\phi) =d_0 > 0$, \eqref{CH_DG} can be seen as the \textit{dual} version of the mixed CG setup \eqref{CG_discr} with respect to the aforementioned discrete de-Rham complex
\begin{equation}
    \xymatrix{
      \mathbb{V}_{\text{CG}} \ar[r]^{\nabla}  &\mathbb{V}_{H(\text{curl})} \ar[r]^{\nabla \times} & \mathbb{V}_{H(\text{div})} \ar[r]^{\nabla \cdot}  & \mathbb{V}_{\text{DG}}}. \label{dR_complex}
\end{equation}
To uncover the duality, note that if $\phi_h$ and $\mu_h$ are elements of $\mathbb{V}_{\text{CG}}$, then their gradients $\nabla \phi_h$ and $\nabla \mu_h$ are elements of $\mathbb{V}_{H(\text{curl})}$. The mixed CG setup can therefore equivalently be rewritten as
\begingroup
\addtolength{\jot}{2mm}
\begin{subequations} \label{CG_coupling_js}
\begin{align}
&\left\langle \eta, \pp{\phi_h}{t} \right\rangle + \left\langle \nabla \eta, \mathbf{j}_h \right\rangle = \langle \eta, S \rangle, & \forall \eta \in \mathbb{V}_{\text{CG}},\\
&\left\langle \mathbf{w}, \tfrac{1}{d_0}\mathbf{j}_h - \nabla \mu_h \right\rangle = 0 &\forall \mathbf{w} \in \mathbb{V}_{H(\text{curl})}, \\
&\left\langle \eta, \mu_h \right\rangle = \left\langle \eta, \phi_h^3 - \phi_h \right\rangle + \left\langle \nabla \eta, \epsilon^2 \bs{\sigma}_h \right\rangle , & \forall \eta \in \mathbb{V}_{\text{CG}}, \\
&\langle \mathbf{w}, \bs{\sigma}_h - \nabla \phi_h \rangle = 0 &\forall \mathbf{w} \in \mathbb{V}_{H(\text{curl})},
\end{align}
\end{subequations}
\endgroup
which is dual to the DG discretization \eqref{CH_DG} in the sense that we consider the opposite side of the de-Rham complex, with $(\mathbb{V}_{H(\text{div})}, \mathbb{V}_{\text{DG}})$ instead of $(\mathbb{V}_{\text{CG}}, \mathbb{V}_{H(\text{curl})})$ -- and therefore swapped occurrences of weak and strong applications of the gradient and divergence operators, as well as strongly applied boundary conditions rather than weakly. Note that this duality holds only on the level of the de-Rham complex, and does not imply that the two discretizations are equivalent, with differences e.g. in local versus global mass conservation, and different convergence rates analogously to the case for the CG and mixed DG discretizations for the Poisson equation.
\end{remark}

In the following proposition, we demonstrate the novel spatial discretization's structure-preserving properties. For this purpose, we define strong consistency of a discrete equation or property in the usual sense, such that the underlying non-discretized equation's exact solution also satisfies the discrete equation or property.
\begin{proposition} \label{prop_DG_ec}
Up to the forcing term $S$, the mixed-DG spatial discretization \eqref{CH_DG} conserves total mass. Further, its discrete total energy, given by
\begin{equation}
H_{D\!G}(\phi_h) \coloneqq \int_\Omega \left(F(\phi_h) + \frac{\epsilon^2}{2} |\bs{\sigma}_h(\phi_h)|^2\right)dx, \label{CH_energy_dg}
\end{equation}
is such that its rate of change of time is strongly consistent.
\end{proposition}
\begin{proof}
To show total mass conservation, we set $\psi \equiv 1$ in \eqref{CH_dg_phi} and obtain
\begin{align}
\frac{d}{dt}\int_\Omega \phi_h \; dx = \int_\Omega \left(S - \nabla \cdot \mathbf{j}_h\right) dx = \int_\Omega S \; dS,
\end{align}
where in the second equality, we applied integration by parts cell-wise for the divergence. The resulting integrals for facets in the given mesh's interior then vanish since $\mathbf{j} \in \mathring{\mathbb{V}}_{H_{\text{div}}}$ has continuous normal components. Further, the integrals for boundary facets vanish due to homogeneous boundary conditions incorporated into $\mathring{\mathbb{V}}_{H_{\text{div}}}$.

To show strong consistency of the total energy evolution, we first note that $\bs{\sigma}_h$ as defined by \eqref{sigma_dg} corresponds to a consistent discretization of the discrete phase field gradient, and therefore $H_{D\!G}$ defined by \eqref{CH_energy_dg} is consistent with the non-discretized total energy functional $H_{C\!H}$ \eqref{CH_energy}. Next, we find that its variational derivative is given by
\begingroup
\addtolength{\jot}{2mm}
\begin{align}
\left\langle \psi, \dd{H_{D\!G}}{\phi_h}\right\rangle \coloneqq& \lim_{\alpha \rightarrow 0} \frac{1}{\alpha}\left(H_{D\!G}(\phi_h + \alpha \psi) - H_{D\!G}(\phi_h)\right) = \left\langle \psi, \dd{F}{\phi_h} \right\rangle + \left\langle \bs{\sigma}_h(\psi), \epsilon^2 \bs{\sigma}_h(\phi_h) \right\rangle \nonumber \\
=& \left\langle \psi, \dd{F}{\phi_h} \right\rangle - \left\langle \psi, \epsilon^2 \nabla \cdot \bs{\sigma}_h(\phi_h)\right\rangle = \langle \psi, \mu_h \rangle &\hspace{-1cm} \forall \psi \in \mathbb{V}_{\text{DG}}, \label{var_form_H_dg}
\end{align}
\endgroup
where in the second to last equality, we used the definition \eqref{sigma_dg} of $\bs{\sigma}_h$ as a general weak gradient operator applied to $\psi$, with test function $\mathbf{v} = \epsilon^2\bs{\sigma}_h(\phi_h) \in \mathring{\mathbb{V}}_{H(\text{div})}$. Additionally, we applied the definition \eqref{mu_dg} of $\mu_h$ in the last equality, noting that analogously to the non-discretized case, we have
\begin{align}
\left\langle \psi, \dd{F}{\phi_h} \right\rangle = \left\langle \psi, \phi_h^3 - \phi_h \right\rangle &&\forall \psi \in \mathbb{V}_{\text{DG}}.
\end{align}
Proceeding as in the non-discretized case \eqref{HCH_dt}, we then obtain
\begingroup
\addtolength{\jot}{2mm}
\begin{align}
\frac{dH_{D\!G}}{dt} = &\left\langle \dd{H_{D\!G}}{\phi_h}, \pp{\phi_h}{t} \right\rangle = \left\langle \mu_h, \pp{\phi_h}{t} \right\rangle = - \left\langle \mu_h, \nabla \cdot \mathbf{j}_h\right\rangle + \left\langle \mu_h, S \right\rangle = - \left\| \tfrac{1}{\sqrt{d(\phi)}} \mathbf{j}_h\right\|_2^2 \!\!+ \left\langle \mu_h, S \right\rangle, \label{Hdg_dt}
\end{align}
\endgroup
where we used the definition \eqref{j_dg} of $\mathbf{j}_h$ with test function $\mathbf{v} = \mathbf{j}_h$. Finally, we have that before discretization, the (negative) flux is defined by $\mathbf{j} = - d(\phi)\nabla \mu$, and so the non-discretized form \eqref{HCH_dt} of the rate of change of total energy equally reads
\begin{align}
\frac{dH_{C\!H}}{dt} = - \left\|\sqrt{d(\phi)}\nabla \mu\right\|_2^2 + \left\langle \mu, S \right\rangle = - \left\|- \frac{1}{\sqrt{d(\phi)}} \mathbf{j} \right\|_2^2 + \left\langle \mu, S \right\rangle = - \left\| \frac{1}{\sqrt{d(\phi)}} \mathbf{j} \right\|_2^2 + \left\langle \mu, S \right\rangle.
\end{align}
\end{proof}

Before proceeding to the case including fluid flow, we include a remark on the previous proposition's mass conservation property. As indicated before, the discrete mass of the spatial discretization \eqref{CH_DG} is conserved locally (when excluding the forcing term $S$), since $\nabla \cdot \mathbf{j}_h \in \mathbb{V}_{\text{DG}}$ and therefore $\pp{\phi_h}{t} + \nabla \cdot \mathbf{j}_h \equiv 0$ as an $L^2$ function.

\textbf{Coupling to fluid equations}. In the presence of fluid flow, the advection term $\nabla \cdot (\mathbf{u} \phi)$ can be discretized readily using e.g., classical DG upwinding \cite{kuzmin2010guide}, which adds the necessary transport stabilization without amending any of the Cahn-Hilliard equation's remaining terms.
%
\begin{definition} \label{def_dg_flow}
Consider DG and div-conforming spaces as described in Definition \ref{def_dg}. Then the mixed DG based discretization for the Cahn-Hilliard equation coupled to flow \eqref{CH_cts_flow}, with homogeneous Neumann boundary conditions for $\phi$ and $\mu$, as well as free-slip boundary conditions for $\mathbf{u}_h$, is given by finding $(\phi_h, \mathbf{j}_h, \mu_h, \bs{\sigma}_h) \in (\mathbb{V}_{\text{DG}}, \mathring{\mathbb{V}}_{H(\text{div})}, \mathbb{V}_{\text{DG}}, \mathring{\mathbb{V}}_{H(\text{div})})$, such that
\begingroup
\addtolength{\jot}{2mm}
\begin{align}
&\left\langle \psi, \pp{\phi_h}{t} + \nabla \cdot \mathbf{j}_h \right\rangle - \langle \nabla_h \psi, \mathbf{u}_h \phi_h \rangle + \int_\Gamma \left[\!\left[\mathbf{u}_h \psi \right]\!\right] \tilde{\phi}_h \; dS = \langle \psi, S\rangle &\forall \psi \in \mathbb{V}_{\text{DG}}, \label{CH_DG_flow}
\end{align}
\endgroup
and equations for $\mathbf{j}_h$, $\mu_h$, and $\bs{\sigma}_h$ as defined by \eqref{j_dg}, \eqref{mu_dg}, and \eqref{sigma_dg}, respectively, in Definition \ref{def_dg}. $\nabla_h$ denotes the element-wise gradient, $\Gamma$ the set of all interior facets of the mesh, and we applied jump, and upwind facet operations defined by
\begin{equation}
[\![\mathbf{v}]\!] \coloneqq \mathbf{v}^+ \cdot \mathbf{n}^+ + \mathbf{v}^- \cdot \mathbf{n}^-, \;\;\; \tilde{\psi} \coloneqq s(\mathbf{u}_h)^+ \psi^+ + s(\mathbf{u}_h)^- \psi^-, \;\;\; \text{for} \;\;\; s(\mathbf{u}_h) = \tfrac{1}{2} (\text{\rm sign}(\mathbf{u}_h \cdot \mathbf{n})+1) \label{def_upw}
\end{equation}
for any vector field $\mathbf{v}$ and scalar field $\psi$. Here, $\mathbf{n}$ denotes the facet normal vector, and the two sides of each mesh facet are arbitrarily denoted by $+$ and $-$ (and hence $\mathbf{n}^+ = - \mathbf{n}^-$).
\end{definition}
%
\begin{proposition}
Assume that the discrete velocity field $\mathbf{u}_h$ is defined in a space $\mathbb{V}_{u_h}$ of vector functions whose normal component across interior facets is continuous and vanishes at boundary facets. Then up to the forcing term $S$, the spatial discretization \eqref{CH_DG_flow} conserves total mass.
\end{proposition}
\begin{proof}
To show total mass conservation, we set $\psi \equiv 1$ as in Proposition 1, and obtain
\begin{align}
\frac{d}{dt}\int_\Omega \phi_h \; dx = \int_\Omega \left(S - \nabla \cdot \mathbf{j}_h\right) dx - \langle \nabla_h 1, \mathbf{u}_h \phi \rangle + \int_\Gamma \left[\!\left[\mathbf{u}_h \right]\!\right] \tilde{\phi}_h \; dS = \int_\Omega S \; dS,
\end{align}
where we proceeded as before for $\nabla \cdot \mathbf{j}_h$. Further, the element-wise gradient on $1$ evaluates to zero. Finally, we have $[\![\mathbf{u}]\!] \equiv 0$ given our assumption on the discrete velocity space $\mathbb{V}_u$.
\end{proof}

Note that the assumption made in Proposition 2 on $\mathbb{V}_{u_h}$ holds true e.g. for $\mathbb{V}_{u_h} = \mathring{\mathbb{V}}_{H(\text{div})}$, or for the continuous vector space $\mathbb{V}_{u_h} = \left(\mathbb{V}_{\text{CG}}\right)^n$, for dimension $n$ of our domain $\Omega$.
%
\begin{remark}
While in this work, we only consider a predefined velocity field $\mathbf{u}_h$, we here briefly remark on the coupling to the moment equation. Using a suitable discretization for the forcing term $\frac{\phi}{\rho}\nabla \mu$ occurring in the momentum equation \eqref{u_eqn_cts}, it is still possible to retain energetic consistency in the discrete energy coupling, as was the case for the CG discretization \eqref{CG_coupling}. For this purpose, we define
\begingroup
\addtolength{\jot}{2mm}
\begin{subequations} \label{DG_coupling}
\begin{align}
&\left\langle \psi, \pp{\phi_h}{t} + \nabla \cdot \mathbf{j}_h \right\rangle - \langle \nabla_h \psi, \mathbf{u}_h \phi_h \rangle + \int_\Gamma \left[\!\left[\mathbf{u}_h \psi \right]\!\right] \tilde{\phi}_h \; dS = \langle \psi, S\rangle &\forall \psi \in \mathbb{V}_{\text{DG}}, \\
&\left\langle \rho_h \mathbf{w}, \pp{\mathbf{u}_h}{t} \right\rangle + \left\langle \nabla \mu_h, \mathbf{w} \phi_h \right\rangle - \int_\Gamma \left[\!\left[\mathbf{w} \mu_h \right]\!\right] \tilde{\phi}_h \; dS = 0 & \forall \mathbf{u} \in \mathbb{V}_{u_h}.
\end{align}
As for the CG case, energetic consistency then follows analogously to the continuous case \eqref{H_consistent_cts}, by setting $\psi = \mu_h$ and $\mathbf{w} = \mathbf{u}_h$ in the phase field and momentum equations, respectively. Such an approach for energetically consistent coupling terms in the momentum equation has e.g., been shown to work well in the context of the (thermal) shallow water equations in \cite{eldred2019quasi,wimmer2020energy}.
\end{subequations}
\endgroup
\end{remark}
%
%
\subsection{Adaptive time discretization} \label{sec_time}
In order to discretize the spatial discretizations introduced in Definitions \ref{def_dg} and \ref{def_dg_flow} above in time, let our space-discretized nonlinear ODE in time take the general form
\begin{align}
    \frac{\partial y}{\partial t} = \mathcal{N}(y,t). \label{N_ODE}
\end{align}
Here we consider $s$-stage diagonally implicit Runge--Kutta methods taking the form
\begingroup
\begin{subequations} \label{RK_steps}
\begin{align}
    Y_{(i)} &= y^n + \delta t \sum_{j=1}^i a_{ij} \mathcal{N}(Y_{(j)},t^n+{c}_j\delta t)\hspace{4ex}\text{for }i=1,..,s,\\
    y^{n+1} &= y^n + \delta t \sum_{i=1}^s b_i \mathcal{N}(Y_{(i)},t^n+{c}_i\delta t), \\
    \hat{y}^{n+1} &= y^n + \delta t \sum_{i=1}^s \hat{b}_i \mathcal{N}(Y_{(i)},t^n+{c}_i\delta t),
\end{align}
\end{subequations}
\endgroup
for time step $\delta t$, time $t^n = n \delta t$ for $n \in \mathbb{N}^*$, known state $y^n$ at time $t^n$, unknown state $y^{n+1}$ to be solved for, and embedded solution $\hat{y}^{n+1}$ computed for the time step update to be described further below. We specifically utilize the second order L-stable implicit Runge-Kutta (RK) method TR-BDF2 \cite{bank1985transient,hosea1996analysis} due to its stiff accuracy, L-stability, high (2nd) stage order, and first-order embedded method. The TR-BDF method with parameter $\gamma = 2 - \sqrt{2}$ has $s=3$ stages and is given by Butcher tableaux
\begin{align}
\begin{array}
{c|c}
\mathbf{c} & A\\
\hline
&\mathbf{b} \\
&\hat{\mathbf{b}}
\end{array}
\hspace{1cm}
=
\hspace{1cm}
\begin{array}{c | c c c}
    0 & 0 & 0 & 0 \\ 
    2-\sqrt{2} & 1-\frac{1}{\sqrt{2}} & 1-\frac{1}{\sqrt{2}} & 0 \\
    1 & \frac{1}{2\sqrt{2}} & \frac{1}{2\sqrt{2}} & 1-\frac{1}{\sqrt{2}} \\\hline
      & \frac{1}{2\sqrt{2}} & \frac{1}{2\sqrt{2}} & 1-\frac{1}{\sqrt{2}} \\\hline
      & \frac{1}{3}-\frac{1}{6\sqrt{2}} & \frac{1}{2\sqrt{2}}+\frac{1}{3} & \frac{1}{3}-\frac{1}{3\sqrt{2}}
\end{array}. \label{butcher_tableau}
\end{align}
The matrix $A$ contains the RK coefficients $a_{ij}$, $\mathbf{c}$ corresponds to quadrature points in time within a time-step, $\mathbf{b}$ to weights for the new solution, and $\hat{\mathbf{b}}$ to weights for the embedded solution. In the context of our finite element-based spatial discretization for the Cahn-Hilliard equation, \eqref{N_ODE} reads
\begin{equation}
\pp{\bs{\phi}_c}{t} = M_{DG}^{-1} \mathcal{B}(\bs{\phi_c}),
\end{equation}
for degree-of-freedom coefficient vector $\bs{\phi}_c$ of $\phi_h \in \mathbb{V}_{\text{DG}}$, DG mass matrix $M_{DG}$, and right-hand side operator $\mathcal{B}$ obtained from a nonlinear variational form $b(\psi, \phi)$ defined by
\begingroup
\addtolength{\jot}{2mm}
\begin{subequations} \label{b_nonl}
\begin{align}
& b(\psi, \phi) = -\left\langle \psi, \nabla \cdot \mathbf{j}_h \right\rangle  + \langle \nabla_h \psi, \mathbf{u}_h \phi_h \rangle - \int_\Gamma \left[\!\left[\mathbf{u}_h \psi \right]\!\right] \tilde{\phi}_h \; dS + \langle \psi, S\rangle &\forall \psi \in \mathbb{V}_{\text{DG}}, \\
&\left\langle \tfrac{1}{\bar{d}(\phi_h)}\mathbf{v}, \mathbf{j}_h\right\rangle - \frac{d_0}{c_0} \left\langle \nabla \cdot \mathbf{v}, \mu_h \right\rangle = 0 &\forall \mathbf{v} \in \mathring{\mathbb{V}}_{H(\text{div})},\\
&\left\langle \psi, \mu_h\right\rangle  - c_0 \left\langle\psi,  \phi_h^3 - \phi_h - \epsilon^2 \nabla \cdot \bs{\sigma}_h \right\rangle = 0 &\forall \psi \in \mathbb{V}_{\text{DG}},\\
&\langle \mathbf{v}, \bs{\sigma}_h \rangle + \langle \nabla \cdot \mathbf{v}, \phi_h \rangle = 0 &\forall \mathbf{v} \in \mathring{\mathbb{V}}_{H(\text{div})},
\end{align}
\end{subequations}
\endgroup
where $c_0 > 0$ is a free scaling parameter set according to solver considerations to be further explained in Section \ref{sec_solver} below. Note that $c_0$ is consistent in the sense that in the strong form when substituting $\mu$ into the $\mathbf{j}$-equation, the occurrences of $c_0$ and $1/c_0$ cancel. For the same solver related reasoning, we additionally split the diffusion coefficient $d(\phi_h) = d_0 \bar{d}(\phi_h)$ into a constant value $d_0$ corresponding to a magnitude, and a normalized varying component $\bar{d}(\phi)$. In particular, for constant diffusion, we simply have $d(\phi_h) = d_0$, and $\bar{d}(\phi) \equiv 1$. Finally, we note that in our discretization of the CH equation, we only evolve the field $\phi_h$, while $\mathbf{j}_h$, $\mu_h$, and $\bs{\sigma}_h$ are auxiliary fields. From a time discretization point of view, the latter fields are not considered as separate fields evolving in time, and are instead solved anew in each nonlinear iteration within a given RK stage, in order to compute the stage's contribution to $\bs{\phi}_c$.

In practice, in order to compute for the RK steps \eqref{RK_steps}, we solve for the implicit nonlinear system of equations
\begin{equation}
R(\bs{\phi}_c^*) = M_{D\!G}(\bs{\phi}_c^* - \bs{\phi}_{rhs}) - \alpha\delta t \mathcal{B}(\bs{\phi}_c^*) = 0, \label{rk_stage_field}
\end{equation}
where $\bs{\phi}_c^*$ is the unknown degree-of-freedom coefficient vector to be solved for, $\bs{\phi}_{rhs}$ is a known degree-of-freedom coefficient vector, and $\alpha$ is a known coefficient from the Butcher tableau. Finally, we also require an explicit solve for the RK scheme's first stage, corresponding to mass matrix inversions.

\textbf{Adaptivity.} The CH equation is often characterized by short periods of rapid activity, followed by longer ones of slow change. To efficiently simulate the evolution of phase separation, it is therefore useful to consider adaptive time stepping methods \cite{soderlind2002automatic,soderlind2006adaptive}. In the context of embedded Runge--Kutta methods such as TR-BDF2, each time step outputs a solution of order $p$ and $p-1$ (for TR-BDF2 we have $p=2$), and the two solutions are compared to provide a local error estimate, which is then used to guide time-step choice given a specified error tolerance.

Here we follow the framework developed in \cite{soderlind2002automatic,soderlind2006adaptive}, and practical algorithm and parameterization detailed in \cite{ranocha2022optimized}, which proceeds as follows. Once the RK time step has been computed using $\delta t$ for updated solution $y^{n+1}$ and embedded solution $\hat{y}^{n+1}$, we define the local error estimate at entry $i$ of our finite element coefficient vectors as
\begin{equation}
    [\delta_{n+1}]_i = \frac{[y^{n+1}]_i - [\hat{y}^{n+1}]_i}
        {\textnormal{tol}_a + \textnormal{tol}_r \max\{ [y^{n+1}]_i, [\hat{y}^{n+1}]_i\}},
\end{equation}
for absolute and relative tolerances $\textnormal{tol}_a$ and $\textnormal{tol}_r$, respectively. This error measure weights each point uniformly and relative to its local solution magnitude, thus making sure that solution values close to zero or refined mesh regions of interest are not weighted less. A global error estimate for the solution at time $t^{n+1}$ is then obtained as
\begin{equation}
    \varepsilon_{n+1} \coloneqq \|\bs{\delta}_{n+1}\| = \left(\frac{1}{N}\sum_{i=1}^N ([\delta_{n+1}]_i)^2\right)^{\frac{1}{2}}.
\end{equation}
We set $\epsilon_0 = 0$ at the initial time $t^0$.

Given time step $\delta t$ used to compute $\epsilon_{n+1}$, a new time step is then suggested based on the error, fitting some polynomial through error of previous steps. A general scaling and corresponding new time step can be written as
\begin{equation}
    \rho = s_0 \; \varepsilon_{n+1}^{-\beta_1/p} \varepsilon_{n}^{-\beta_2/p}, \hspace{1cm} \hat{\rho} = 1 + \kappa \arctan\left( \frac{\rho - 1}{\kappa}\right), \hspace{1cm} \delta t \mapsfrom \hat{\rho} \delta t, \label{dt_new}
\end{equation}
where $\hat{\rho}_{m+1}$ is introduced to limit the change in step size from being overly rapid. In the numerical results section below, we set
\begin{equation}
\textnormal{tol}_a = 10^{-4}, \;\;\; \textnormal{tol}_r = 10^{-5}, \;\;\; s_0 = 0.9, \;\;\; \beta_1 = 0.4, \;\;\; \beta_2 = -0.2, \;\;\; \kappa = 2. \label{dt_params}
\end{equation}

If $\varepsilon_{n+1} > 1$, we reject the time step, and retry the RK step from time $t^n$ with the suggested reduced time step $\delta t\mapsfrom \hat{\rho}\delta t$ (the arctan scaling and parameters in \eqref{dt_params} are used to smooth the evolution of $\delta t$ and avoid failed timesteps due to excessive increases in $\delta t$, as failed timesteps can rapidly increase computational cost). Otherwise, if $\varepsilon_{n+1} < 1$, we accept the step at time $t^n$ and set the first time step to be used at time $t^{n+1}$ equal to the proposed adjusted $\delta t \mapsfrom \min \{\delta t_{max}, \hat{\rho}\delta t\}$ as computed in \eqref{dt_new}. In the numerical results section below, where we consider non-dimensionalized test cases, we consider a maximum time step of $\delta t_{max} = 12$.
\subsection{Preconditioning} \label{sec_solver}
The implicit nonlinear problem \eqref{rk_stage_field} in our adaptive implicit Runge-Kutta method relies on the form $b(\psi, \phi)$ given by \eqref{b_nonl}. Applying a Newton iteration approach, the $k^{th}$ iteration of the nonlinear solve is given by
\begin{equation}
- \left[R(\bs{\phi}^*_{c,k}), \bs{0}, \bs{0}, \bs{0}\right]^T = J^{-1}_{DG}\left[\delta \bs{\phi}^*, \mathbf{j}_c, \bs{\mu}_c, 
\bs{\sigma}_c \right]^T,
\end{equation}
for $R$ as given in \eqref{rk_stage_field}, and update $\delta \bs{\phi}_c^* = \bs{\phi}^*_{c, k+1} - \bs{\phi}^*_{c, k}$ based on coffecient vectors $\bs{\phi}^*_{c, k+1}$, $\bs{\phi}^*_{c, k}$ that correspond to the unknown vector to be solved for and known guess for $\bs{\phi}_c^*$ in \eqref{rk_stage_field}, respectively. Similarly, $\mathbf{j}_c, \bs{\mu}_c, \bs{\sigma}_c$ denote the auxiliary variable's degree-of-freedom coefficient vectors. Further, the (non-approximated) Jacobian $J_{DG}(\bs{\phi}^*_{c, k})$ is defined by
\begin{equation}
J_{DG}^{-1}\left[\delta \bs{\phi}_c^*, \mathbf{j}_c, \bs{\mu}_c, 
\bs{\sigma}_c \right]^T =
\begin{bmatrix}
M_{DG} + \alpha \delta t T'_u & \alpha \delta t D & 0& 0 \\
0 & M_{div} & - \frac{d_0}{c_0} D^T& 0 \\
-c_0 F'' & 0 & M_{DG} & c_0 \epsilon^2 D \\
D^T & 0 & 0& M_{div}
\end{bmatrix}^{-1}
\begin{bmatrix}
\delta \bs{\phi}_c^*\\
\mathbf{j}_c\\
\bs{\mu}_c\\ 
\bs{\sigma}_c
\end{bmatrix}, \label{J_non_swap}
\end{equation}
based on the form $b(\psi, \phi)$ defined in \eqref{b_nonl}. Note that we have made use of our assumption of a constant diffusion coefficient, such that $\bar{d}(\phi) \equiv 1$. As before, $M_{DG}$ corresponds to the DG mass matrix and similarly, $M_{div}$ corresponds to a div-conforming space mass matrix. Additionally,
\begin{equation}
F'' = F''(\phi^k) = 3M_{DG,{{\phi_h^k}^2}} - M_{DG}
\end{equation}
denotes the functional derivative of $F'$, where $M_{DG, {{\phi_h^{k}}^2}}$ is given by a DG mass matrix weighted by $\left(\phi_h^k\right)^2$, for DG field $\phi_h^k$ corresponding to the coefficient vector $\bs{\phi}^*_{c,k}$. Further, $D$ in \eqref{J_non_swap} denotes the discrete divergence operator. Finally, $T'_u$ denotes the discrete upwind transport operator, which is based on the corresponding form in \eqref{CH_DG_flow}, but applied to $\delta \phi_h = \phi_h^{k+1} - \phi_h^k$, with $\phi_h^{k+1}$ defined analogously to $\phi_h^k$.

In order to proceed with our preconditioning strategy, we first rearrange the Jacobian by swapping rows and columns to arrive at the equivalent block linear system
\begin{equation}
-\left[\bs{0}, \bs{0}, \bs{0}, R(\bs{\phi}^*_{c,k})\right] =
\begin{bmatrix}
M_{div} & 0 & 0 & - \frac{d_0}{c_0} D^T\\
0 & M_{div} & D^T &0\\
0 & c_0 \epsilon^2 D & -c_0 F'' & M_{DG}\\
\alpha \delta t D & 0 & M_{DG} + \alpha \delta t T'_u & 0
\end{bmatrix}^{-1}
\begin{bmatrix}
\mathbf{j}_c\\
\bs{\sigma}_c\\
\delta \bs{\phi}_c^*\\
\bs{\mu}_c\\ 
\end{bmatrix}. \label{J_swapped}
\end{equation}
To solve for the row-swapped Jacobian inverse in \eqref{J_swapped}, we then consider the following preconditioning approach. If we eliminate the top $2\!\times\!2$ block (that is we eliminate the auxiliary fluxes $\mathbf{j}_h$ and $\bs{\sigma}_h)$, we arrive at a $2\!\times\!2$ block Schur complement of the form
\begin{equation}
S_{\phi\mu} =
\begin{bmatrix}
- c_0 \left(F'' + \epsilon^2 D M_{div}^{-1}D^T\right) & M_{DG}\vspace{2mm}\\
M_{DG} + \alpha \delta t T'_u & \frac{\alpha \delta t d_0}{c_0} D M_{div}^{-1}D^T
\end{bmatrix}. \label{Schur_phimu}
\end{equation}
Note that $L_{DG} = DM^{-1}_{div}D^T$ is a discrete DG-based diffusion operator, which is nonlocal due to $M^{-1}_{div}$. We thus define $\tilde{L}_{DG}$ to be a suitable sparse approximation to $L_{DG}$; in the numerical results section below, we set $\tilde{L}_{DG}$ to be equal to a DG interior penalty based formulation of the Laplacian \cite{burman2005unified}; for details, see Appendix \ref{app_IPDG}.

We then follow the existing approach for the CG case outlined for the row-swapped system \eqref{phimu_system_swapped}, and precondition $S_{\phi\mu}$ in a similar fashion to \eqref{CG_diag_precon}, using a block diagonal preconditioner. For this purpose, we set our scaling parameter $c_0$ such that the two discrete Laplacians in \eqref{Schur_phimu} scale equally, which we found to be beneficial for our overall solver efficiency. This holds true for
\begin{equation}
c_0 = \frac{\tau}{\epsilon^2}, \hspace{1cm} \tau = \epsilon \sqrt{\alpha \delta t d_0}, \hspace{1cm} S_{\phi\mu} =
\begin{bmatrix}
- \frac{\tau}{\epsilon} F'' - \tau L_{DG} & M_{DG}\vspace{2mm}\\
M_{DG} + \alpha \delta t T'_u & \tau L_{DG}
\end{bmatrix},
\end{equation}
noting that $c_0$ is non-dimensional (since the RK coefficient $\alpha$ is non-dimensional). In this case, both Laplacians are scaled by $\tau$, which has unit m$^2$ as expected, given that the Laplacian $\Delta$ by itself has unit m$^{-2}$. The preconditioner is then given by
\begin{equation}
P_S \coloneqq
\begin{bmatrix}
M_{DG} + \frac{2\tau}{\epsilon}M_{DG} + \tau \tilde{L}_{DG}& 0 \\
0 & M_{DG} + \tau \tilde{L}_{DG} + \alpha \delta t T'_u
\end{bmatrix}. \label{DG_diag_precon}
\end{equation}
Recall that the CG-based preconditioner \eqref{CG_diag_precon} derived in \cite{brenner2018robust} was shown to be robust for an approximate version of the CG-based (row-swapped) Jacobian \eqref{phimu_system_swapped}, in which $F''$ was simplified assuming $\left(\phi_h^{k}\right)^2 \approx 1$. In our case, this simplification reads $c_0F'' = \tfrac{\tau}{\epsilon^2}F'' \approx \frac{2\tau}{\epsilon}M_{DG}$, and unlike in \cite{brenner2018robust}, here we choose to keep this term in our preconditioner's top-left block. In particular, we found it to be important to keep such a term in view of our adaptive time stepping scheme, as it can potentially lead to large time steps and $\tau$ scales with $\sqrt{\delta t}$.

In the numerical results section below, we consider \texttt{gmres} implemented in PETSc \cite{balay2019petsc} for the Schur complement preconditioning based outer solver, using a relative solver tolerance of $10^{-8}$. Further, for the inner solves, we use one iteration of SOR (implemented in parallel in PETSc as block Jacobi with SOR in each block) to approximate applications of $A_{00}^{-1}$, where $A_{00}$ is the top-left $2\!\times\!2$ Hdiv mass matrix block in \eqref{J_swapped}. Applications of ${P_S}^{-1}$ are computed by applying one iteration of classical AMG as implemented boomerAMG in hypre \cite{falgout2002hypre} to each of the diagonal blocks in \eqref{DG_diag_precon}.

\begin{remark}
For a scheme based on the CG discretization \eqref{CG_coupling}, we would proceed analogously and consider a Krylov subspace method for (a suitable scaled form of) the preconditioner \eqref{CG_diag_precon}, again applying classical AMG once to each of the diagonal blocks. Compared to the CG case, the DG case therefore leads to an additional cost by first having to take a Schur complement needed to eliminate the two vector auxiliary fields. This is analogous to the case of discretizing the Poisson equation as described in this section's initial motivation, which can be discretized either in a CG based primal form, or in an Hdiv-DG mixed form.
\end{remark}

\begin{remark}
It can be shown that the Cahn Hilliard equations can be cast in non-dimensional form
\begingroup
\begin{subequations}
\begin{align}
&\pp{\tilde{\phi}}{\tilde{t}} - \frac{1}{\text{Pe}} \tilde{\nabla} \cdot\left( \bar{d}\big(\tilde{\phi}\big) \nabla \tilde{\mu}\right) + \tilde{\nabla} \cdot (\tilde{\mathbf{u}} \tilde{\phi}) = \tilde{S}, \\
&\tilde{\mu} = \tilde{F}'(\tilde{\phi}) - \gamma^2 \Delta \phi,
\end{align}
\end{subequations}
\endgroup
where non-dimensionalized quantities are denoted with a tilde. Pe$=\tfrac{u_0 L_0}{d_0}$ denotes the P\'eclet number, and $\gamma = \epsilon/L_0$. Here, $L_0$ and $u_0 = L_0/t_0$ denote the problem's reference length scale and velocity, respectively, for reference time frame $t_0$. With the choices of resolution, Pe and $\epsilon$ used in the numerical results section below, we found our preconditioning strategy detailed above to work well. However, for very strongly advection dominated cases, the transport operator $\alpha \delta t T'_u$ may dominate over the elliptic operator related terms. In this case, the Schur complement \eqref{Schur_phimu} -- which contains $\alpha \delta t T'_u$ on the off-diagonal -- may no longer be preconditioned well by the diagonal preconditioner \eqref{DG_diag_precon}. In this case, we would skip the row-swapping step, and instead precondition the system using AIR \cite{manteuffel2019nonsymmetric, manteuffel2018nonsymmetric}, a multigrid method designed for transport dominated problems. In particular, AIR is designed considering a reduced matrix connectivity due to information being carried from the upwind direction, which works especially well for DG-upwind based discretizations as considered in this work.
\end{remark}
%
\section{Numerical results} \label{sec_Numerical_results}
Having introduced our novel spatial discretization \eqref{CH_DG} as well as the adaptive time discretization and solver strategy, we move on to presenting numerical results to test the method's order of accuracy, structure-preserving properties, stability with respect to transport, and solver robustness. Further we demonstrate the adaptive time stepping capability, as well as the DG-upwind based setup in the advection dominated regime. The tests were performed  using Firedrake~\cite{FiredrakeUserManual}, which heavily relies on PETSc \cite{balay2019petsc}. We consider quadrilateral or hexahedral meshes for all test cases, with finite element spaces given by the continuous Galerkin second polynomial order space $\mathbb{V}_{\text{CG}} = Q_2$, the second order Raviart-Thomas space $\mathbb{V}_{H(\text{div})} = RT_{c_2}^f$ in 2D as well as the second order N\'ed\'elec space $N_{c_2}^f$ in 3D, and the first order discontinuous Galerkin space $\mathbb{V}_{\text{DG}} = dQ_1$. Note that these spaces form a discrete de-Rham complex (together with the curl-conforming space $N_{c_2}^e$ in 3D). Further, we consider adaptive time stepping parameters given by \eqref{dt_params}, and set $\mathbf{u} \equiv \mathbf{0}$, $S=0$ unless noted otherwise. Finally, the relative nonlinear and linear solver tolerances are each set to $10^{-8}$.

\subsection{Accuracy and solver robustness}
\textbf{Convergence rates.} We first test the novel discretization's convergence, structural properties, as well as stability in the advection dominated regime. For the former, we consider a manufactured solution \cite{roache2002code} on a periodic 2D domain $\Omega = [0, 1]^2$, given by
\begin{equation}
\phi_{\text{IC}} = \sin(2\pi x)\sin(4\pi y). \label{phi_manufactured}
\end{equation}
Further, we set $d = 1$, $\epsilon = 0.1$. The counter-forcing term in \eqref{b_nonl} then reads
\begin{equation}
S = - \nabla \cdot \big(d \nabla (\phi_{\text{IC}}^3 - \phi_{\text{IC}} - \epsilon^2 \Delta \phi_{\text{IC}})\big) \label{counter_forcing}
\end{equation}
and is computed at quadrature points. For $\Omega$, we apply a regular quadrilateral mesh, with resolutions of $16^2$, $32^2$ and $64^2$ cells in the convergence study. We then run for ten time steps with a fixed step of $\delta t = 10^{-3}$, using a direct solver for the resulting linear systems of equations, and consider $L^2$ errors for $\phi_h$ and its discrete gradient $\bs{\sigma}_h$. Further, for comparison we compute the same error quantities using the CG-based mixed spatial discretization \eqref{CG_discr}, as well as the $L^2$ error for $\phi_h$ for a DG interior penalty-based discretization. The latter is presented e.g., in \cite{kay2009discontinuous} and is analogous to the CG-based mixed spatial discretization, except that the weak Laplacian is discretized using an interior penalty method as described in Appendix \ref{app_IPDG}. Here, we set the penalty parameter for the latter method equal to $\kappa = 4$, and results for the three spatial discretizations are given in Figure \ref{ConvStructStab}. We obtain convergence rates akin to the expected ones for the Poisson equation, given by third and second order for $\phi_h \in$ CG and its gradient $\nabla \phi_h$, respectively \cite{brenner2008mathematical}, as well as second order for $\phi_h \in$ DG and its discrete gradient $\bs{\sigma}_h$, respectively \cite{duran2008mixed}. For the interior penalty method, we also obtain a second order rate for $\phi_h$ (higher than its theoretical estimate of $1.5$, likely due to the regular periodic mesh setup).

\textbf{Structure-preserving properties.} We next consider a periodic 3D domain $\Omega = [0, 1]^3$, together with values initialized uniformly at random in the interval $[-1,1]$ for the degrees of freedom of $\phi_h$. This time, we set $d = 0.1$ and $\epsilon = 0.02$. We consider a regular hexahedral mesh with $12^3$ cells, and run up to time $t=4$ with an adaptive time step initialized at $\delta t_0 = 10^{-4}$, using the preconditioning strategy outlined in Section \ref{sec_solver}. The resulting mass error as well as energy and time step evolution are again given in Figure \ref{ConvStructStab}. We find that the mass is conserved up to a relative error of $10^{-11}$, likely related to linear and nonlinear solver tolerances. Further, as expected, the discrete rate of change of energy is always negative. Note that the large changes for mass and energy towards the initial time are due to rapid adjustments to smoothen the initial random distribution according to which our initial field is defined. Finally, as expected, the time step grows over time as the dynamics slow down and larger parts of the domain are either at phase $\phi_h = -1$ or $\phi_h = 1$. The oscillations in time step size are an interplay between increasing time steps at times of slow dynamics and decreasing ones at times of fast dynamics, such as phase absorptions and mergers; a more detailed discussion of this is included further below in this section.

\textbf{Transport stability.} Next, we study the scheme's stabilizing property with respect to advection. For this purpose, we consider a periodic 2D domain $\Omega = [0, 1]^2$, together with initial condition and prescribed flow
\begingroup
\addtolength{\jot}{2mm}
\allowdisplaybreaks
\begin{subequations}
\begin{align}
&\phi_\text{IC} = f(x, \; 0.4)f(y, \; 0.2), \;\;\; \text{for} \;\;\; f(a, a_0) = \frac{1}{2}\left(\tanh\left(\frac{a - a_0}{\sigma_0}\right) - \tanh\left(\frac{a - (1 - a_0)}{\sigma_0}\right)\right), \label{phi_flow}\\
&\mathbf{u} = [u_0, 0]^T, \hspace{1cm} u_0 \in \{0, 1\},
\end{align}
\end{subequations}
\endgroup
where $f$ corresponds to a smoothened 1D step function profile, and we set $\sigma_0 = 0.03$. $\phi_{\text{IC}}$ is then equal to 0 within the domain, except for a rectangle of value 1 at the domain's center, with a thin smooth transition zone from 0 to 1. We set $d = 1/4000$ to ensure a strongly advection dominated regime, and further $\epsilon = 1/100$. Additionally, we consider a perturbed version of a regular quadrilateral mesh with $20^2$ cells\footnote{The perturbation is implemented as a uniformly random factor of up to $\pm 0.06\delta x = 0.06 \tfrac{1}{20}$ for each coordinate corresponding to ``interior'' vertices, i.e. vertices not lying on the computational domain's periodic boundary.}, in order to avoid a reduced $\phi$ error development due to a cell-flow alignment. Finally, we add a counter-forcing term equal to \eqref{counter_forcing} using $\phi_{\text{IC}}(\mathbf{x} - \mathbf{u} t)$, for $\phi_{\text{IC}}$ given by \eqref{phi_flow} above. This then ensures an analytic solution given by $\phi(x,t) = \phi_{\text{IC}}(\mathbf{x} - \mathbf{u} t)$, noting that we consider a periodic domain. The simulation is run with a fixed time step $\delta t = 0.01$ up to time $t=8$, and results depicting the $L^2$ error for $\phi$ are given in Figure \ref{ConvStructStab}, for the mixed CG and DG discretizations \eqref{CG_CH_flow} and \eqref{CH_DG_flow}, respectively. For the case without flow, we find a larger initial projection error into the DG space $dQ_1$ than the CG space $Q_2$, related to the poorly resolved transition zone from phase 0 to 1. Further, there is an increasing error growth over time, since the neutral zone $\phi = 0$ is unstable and small numerical errors eventually force a separation of the latter zone into regions of $\phi = \pm1$. While this happens at a slightly slower rate for the DG method than the CG one, the overall trend is similar. If a background flow is included, the error growth rate for the CG is amplified as expected due to a lack of transport stabilization. For the DG case, the effect of transport is very different: initially, there is a substantial increase in error as the DG upwind stabilization leads to a smoothing effect of the poorly resolved transition zone. However, this is followed by a substantially decreased error growth rate compared to the other runs, as the upwind stabilization acts to suppress instabilities arising in the neutral zone. Overall, the results reflect the underlying discretizations' expected behavior, and qualitatively very similar error developments also hold true for smoother, more well-resolved transition zones (such as for $\sigma_0 = 0.05$; results not shown here).

\begin{figure}[ht]
$\hspace{-3.8mm}$
\begin{minipage}[b]{49mm}
\includegraphics[width=1.67\textwidth]{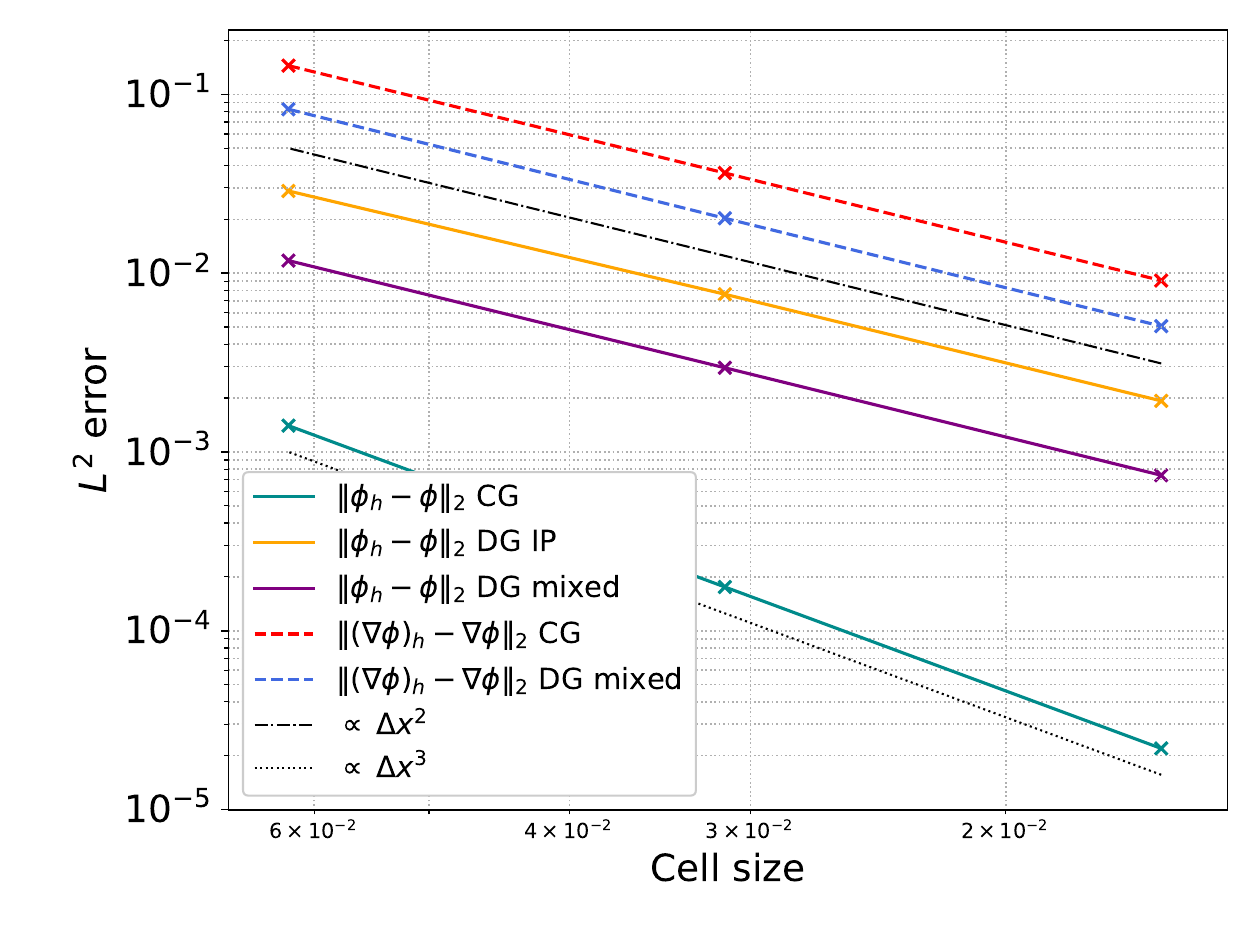}
\vspace{-8mm}\\
\includegraphics[width=1.67\textwidth]{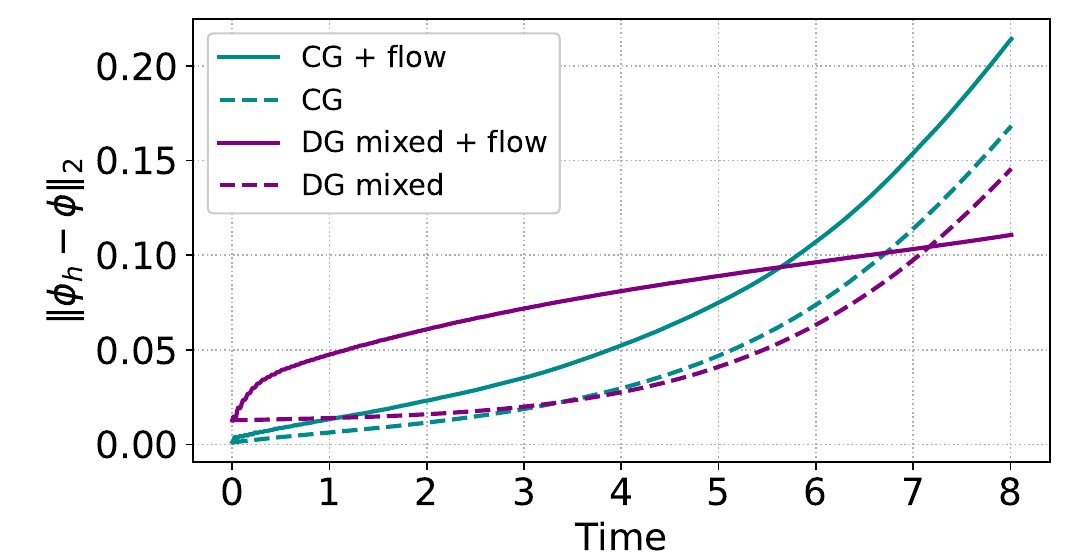}
\end{minipage}
$\hspace{31mm}$
\begin{minipage}[b]{49mm}
\includegraphics[width=1.76\textwidth]{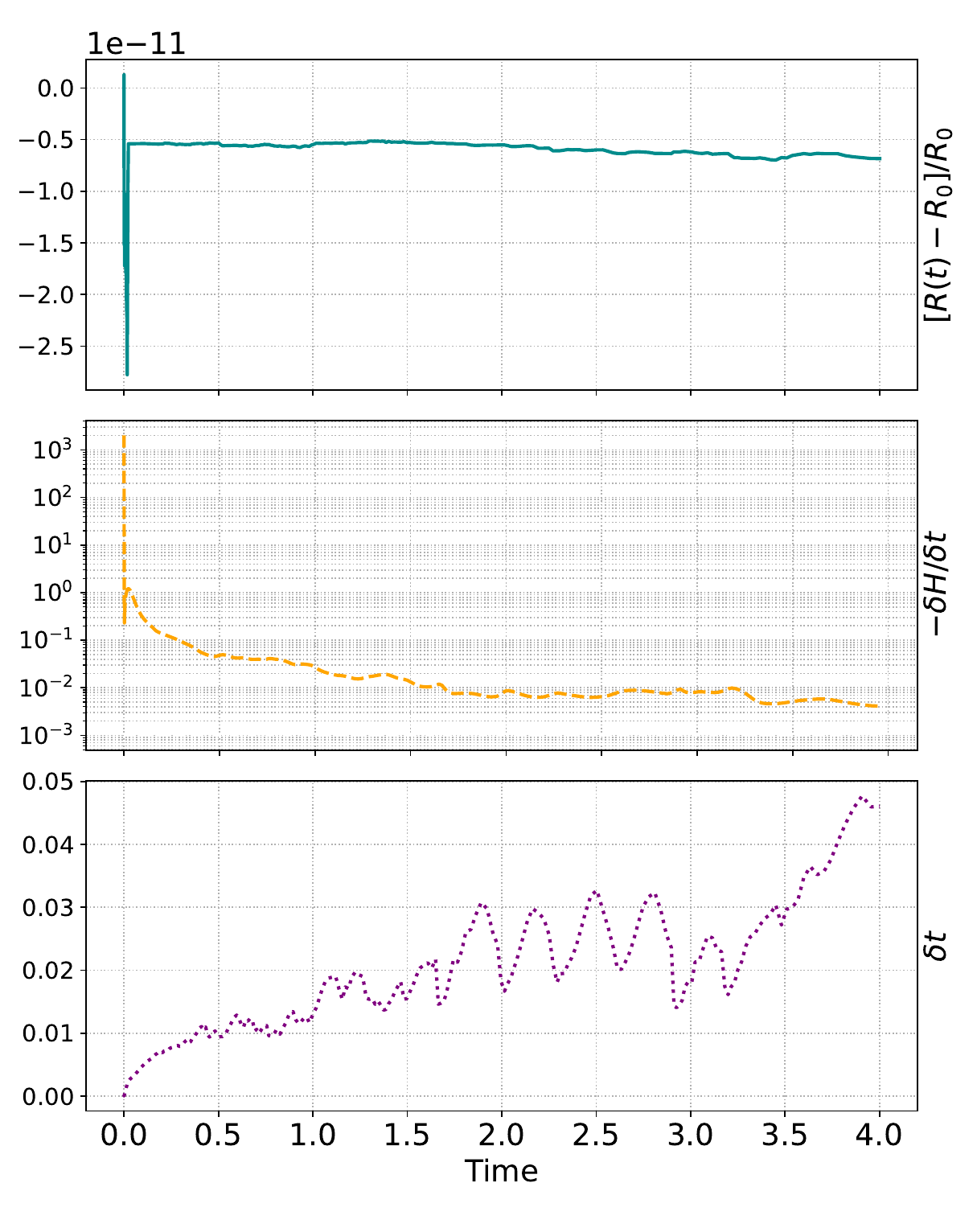}
\vspace{-9.5mm}
\end{minipage}
\vspace{-4mm}
\begin{center}
\caption{Convergence, structural and stability tests. Top left: $L^2$ errors for mixed DG (solid purple for $\phi$, dashed blue for its discrete gradient), mixed CG (solid cyan for $\phi$, dashed red for its gradient), and interior-penalty based DG (solid yellow for $\phi$) discretizations, for manufactured 2D test case \eqref{phi_manufactured}. Right column: evolution in time for relative total mass (top), negative rate of change of total energy (middle), and time step size, for 3D spinodal decomposition test case. Bottom left: $L^2$ errors for mixed DG (purple) and mixed CG (cyan) discretizations, for test case \eqref{phi_flow} with (solid) and without (dotted) flow.} \label{ConvStructStab}
\vspace{-6mm}
\end{center}
\end{figure}
%
\textbf{Solver robustness.} In order to test the proposed solver strategy's robustness, we follow a similar setup to \cite{brenner2018robust}, and consider a periodic unit square domain $\Omega = [0, 1]^2$ together with an initial condition given by
\begin{equation}
\phi_\text{IC} = \tfrac{1}{2}(1 - \cos(4 \pi x))(1 - \cos(2\pi y)) - 1, \label{IC_solver}
\end{equation}
and coefficients $d = 1$, $\epsilon \in \{2^{-4}, 2^{-6}\}$. The mesh is given by a regular quadrilateral one, with $\big[8\times 2^{k_{\delta x}}\big]^2$ cells, for $k_{\delta x} \in \{0, 1, 2, 3, 4, 5\}$, and we run 20 fixed time steps of size $0.002 \times 2^{-k_{\delta t}}$, for $k_{\delta t} \in \{3, 4, 5, 6, 7, 8\}$. The average linear iteration count per nonlinear iteration is given in Table \ref{table_its}, for a) a varying spatial resolution with fixed time step and epsilon, and b) a varying time step with fixed spatial resolution and epsilon.

\begin{table}
\begin{center}
{
\setlength{\tabcolsep}{3.7pt}
\begin{tabular}{|c|c|c|}
\multicolumn{3}{c}{\texttt{gmres} its. for $h_{\delta t} = 6$} \\
\hline
$k_{\delta x}$ & $\epsilon = 2^{-4}$ & $\epsilon = 2^{-6}$ \\
\hline
0 & 46 & 32 \\
1 & 63 & 51 \\
2 & 75 & 72 \\
3 & 65 & 76 \\
4 & 63 & 72 \\
5 & 60 & 65 \\
\hline
\end{tabular}
\hspace{1cm}
\begin{tabular}{|c|c|c|}
\multicolumn{3}{c}{\texttt{gmres} its. for $h_{\delta x} = 3$} \\
\hline
$k_{\delta t}$ & $\epsilon = 2^{-4}$ & $\epsilon = 2^{-6}$ \\
\hline
3 & 64 & 80 \\
4 & 65 & 80 \\
5 & 65 & 78 \\
6 & 65 & 76 \\
7 & 66 & 74 \\
8 & 68 & 70 \\
\hline
\end{tabular}
}
\caption{Average number of linear iterations per nonlinear iteration across implicit RK stages and last ten time steps of simulation of initial conditions \eqref{IC_solver}. The space resolution and time step are given by $\delta x = \tfrac{1}{8} \times 2^{-k_{\delta x}}$, $\delta t = 0.002 \times 2^{-k_{\delta t}}$, respectively.} \label{table_its}
\end{center}
\end{table}

For the fixed time step results, we find a mild growth in iteration count as we increase the spatial refinement level from $0$ to $2$, after which the count stagnates. Further, for the higher refinement levels, we find a mild increase in iteration count as we move from $\epsilon = 2^{-4}$ to $\epsilon = 2^{-6}$. These results are similar to the ones in \cite{brenner2018robust}; in particular we find our solver performance to be robust with respect to spatial refinement for fixed $\epsilon$, and some increase in iteration count as $\epsilon$ gets smaller. For the fixed spatial resolution runs, we again find similar results, given by robustness across time steps for fixed $\epsilon$, and an increase in iteration count as $\epsilon$ gets smaller. Finally, we note that we found similar linear iteration counts across our (implicit) Runge Kutta stages, and a small number of nonlinear iterations ($\sim$1-2) throughout.
%
\subsection{Time adaptivity and advection-dominated regimes} \label{sec_num_dt}
\textbf{Time step adaptivity.} Having verified our novel scheme's accuracy and solver robustness, we next consider more practical test cases. The first is designed to test the adaptive time step and consists of several elliptic ``bubbles'' of phase $\phi = 1$ embedded in an area of phase $\phi = -1$, within a periodic 2D domain. Over time, the elliptic bubbles will deform towards circles in order to reach a lower energetic state. In doing so, they become sufficiently close to each other as to absorb or merge with one another. The deformation occurs on a slow time scale, while the absorptions and mergers occur on a fast one. Therefore, we aim for the scheme to run at a large time step during slow deformation periods, while reducing to smaller ones when the fast time scale dominates the phase field's evolution. Over time, all bubbles merge to one whose diameter is larger than the domain's width. The last bubble then eventually merges with itself across the periodic boundary, forming a strip of phase $\phi = 1$ along the $y$-direction.

The equation's parameters are given by $d = 1/100$, $\epsilon = 1/50$. Further, we consider a periodic domain $\Omega = [0, 2] \times [0, 1]$, and we assume an initial condition expression of the form
\begin{equation}
\phi_{\text{IC}}(\mathbf{x}) =
\begin{cases}
1, \;\;\; \mathbf{x} \in \Omega_1,\\
-1, \;\;\; \text{otherwise},
\end{cases} \;\;\;
\Omega_1 = \sum_{i=1}^5  \big\{(x, y) \in \Omega \colon (x - x_i)^2 + (1-\delta_0)(y - y_i)^2 < r_i^2\big\}, \label{IC_bubbles}
\end{equation}
for elliptic region values
\begin{equation}
\begin{bmatrix}
x_1 = 0.50 & x_2 = 1.00 & x_3 = 1.59 & x_4 = 0.55 & x_5 = 1.20 \\
y_1 = 0.32 & y_2 = 0.65 & y_3 = 0.60 & y_4 = 0.80 & y_5 = 0.17 \\
r_1 = 0.23 & r_2 = 0.25 & r_3 = 0.28 & r_4 = 0.09 & r_5 = 0.11 \\
\end{bmatrix},
\end{equation}
and $\delta_0 = 0.2$. We consider a regular quadrilateral mesh with spatial resolution set to $64 \times 32$ cells, and further the initial time step is set to $\delta t = 10^{-5}$. 
The simulation is run up to $t = 1000$, and results are depicted in Figures \ref{Bubble_t_dev} and \ref{Bubble_phi}.
\begin{figure}[ht]
\begin{center}
\includegraphics[width=0.49\textwidth]{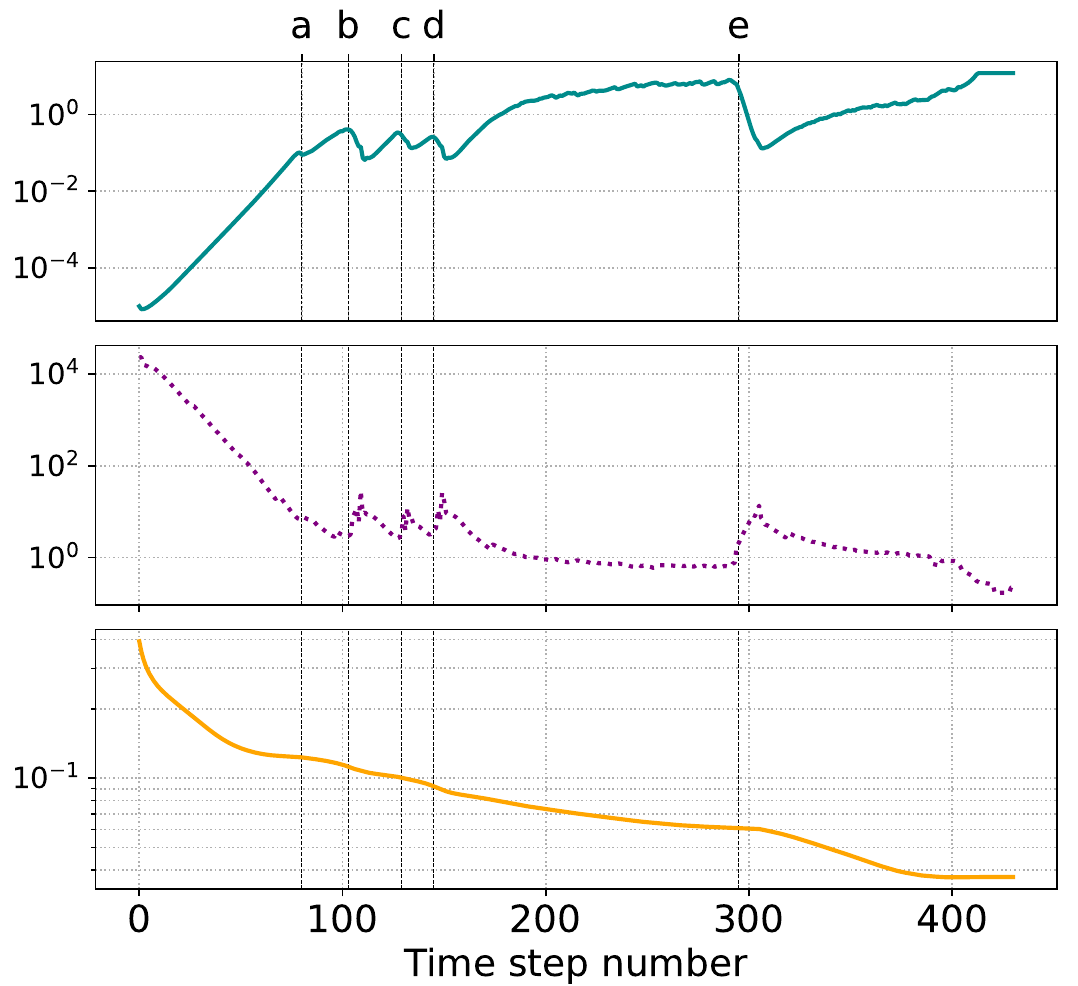}
\includegraphics[width=0.49\textwidth]{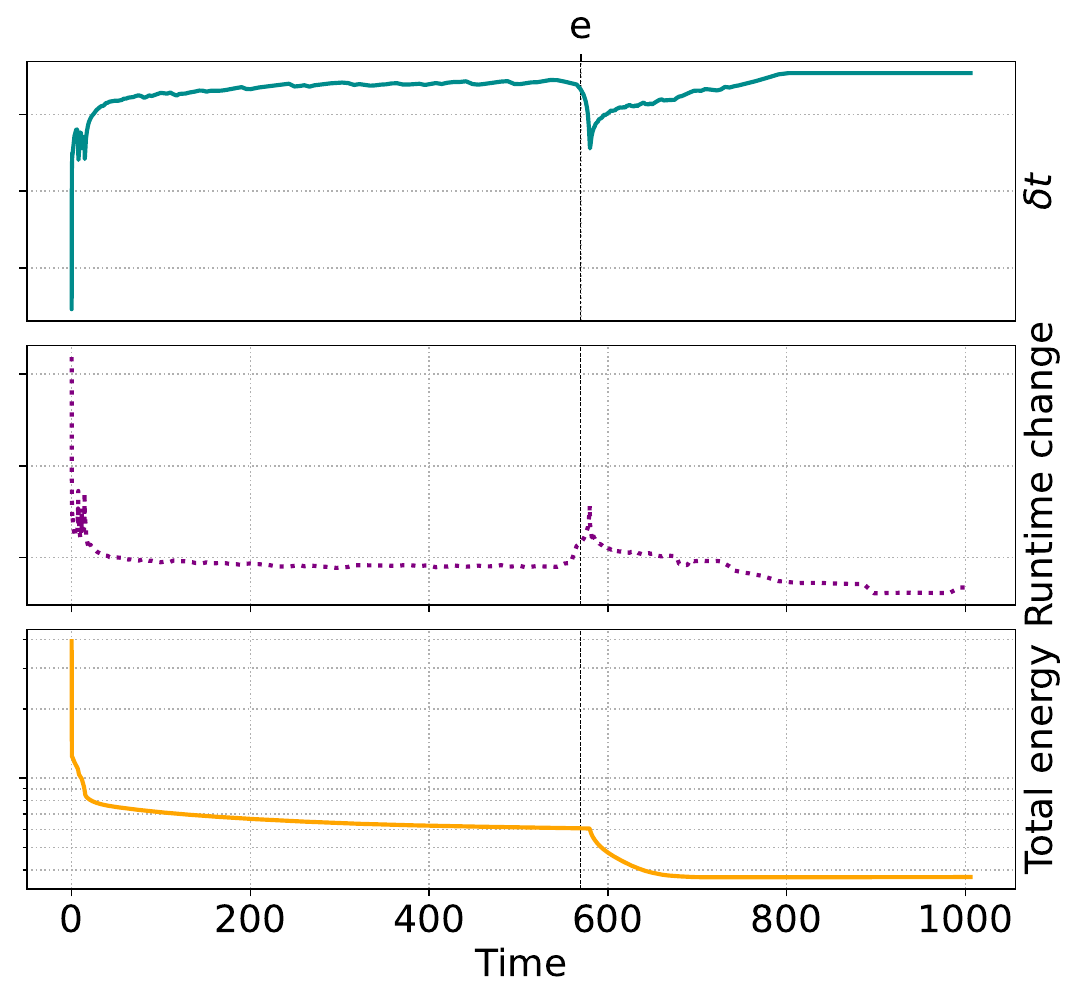}
\vspace{-2mm}
\caption{Evolution of time step (solid green; top), rate of change of cumulative total wall-clock time with respect to simulation time (dotted purple; middle), and total energy (solid orange; bottom), plotted against the number of time steps (left) and simulation time (right). The vertical lines marked from (a) to (e) denote time step number/approximate simulation time $(80/0.91)$, $(103/5.91)$, $(128/10.30)$, $(145/13.67)$ and $(297/575.09)$, respectively.} \label{Bubble_t_dev}
\end{center}
\end{figure}
\begin{figure}[ht]
\begin{center}
\includegraphics[width=.99\textwidth]{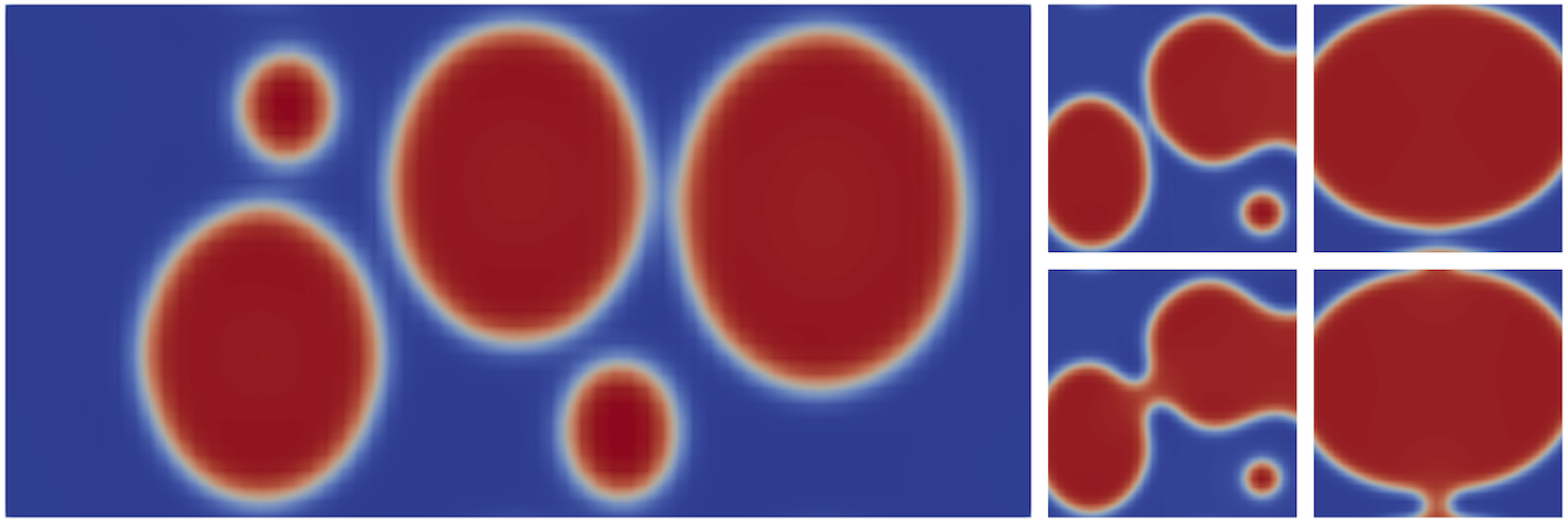}
\vspace{-3mm}
\caption{Left: $\phi$ indicator function at time $t=0.205$, for test case \eqref{IC_bubbles}. Center and right: sections of the domain at times $t=10.30$ (top-center, corresponding to (c) in Figure \ref{Bubble_t_dev}), $t=12.05$ (bottom-center), $t=575.09$ (top-right, corresponding to (e) in Figure \ref{Bubble_t_dev}), and $t=580.06$ (bottom-right). Colors correspond to blue for $\phi = -1$, white for $\phi = 0$, and red for $\phi = 1$.} \label{Bubble_phi}
\end{center}
\end{figure}

Next to periods of growth in time step size, we find five significant reductions, marked from (a) to (e) in Figure \ref{Bubble_t_dev}. Considering the phase field development (see Figure \ref{Bubble_phi}), we find that these correspond to faster evolutions given by (a) the large center and right bubbles merging, (b) the top left bubble being absorbed, (c) the left large bubble merging into the center region, (d) the small bottom bubble being absorbed, and finally (e) the center region merging with itself. Importantly, the time step is reduced ahead of time, before the actual mergers/absorptions occur (see Figure \ref{Bubble_phi} for the mergers (c) and (e)), which ensures that the fast merging dynamics are provided accurate initial conditions and initiate at the correct time. Additionally, the time step varies over 5 orders of magnitude through the simulation, eventually reaching the capped value $\delta t = 12$. This indicates that our adaptive time stepping scheme can not only accurately handle fast time scales, but also effectively step over slow ones.

In terms of wall-clock time, one could consider the change (derivative) in cumulative wall-clock time with respect to physical time. Smaller values then indicate a smaller increase in cumulative wall-clock time for a given stretch of simulation time, i.e. an overall faster runtime. Considering Figure \ref{Bubble_t_dev}, we find that regions of fast-scale evolution lead to both a decrease in time step size and spikes in change in cumulative wall-clock time. This means that as we decrease the time step to resolve fast-scale features, the simulation takes longer to compute, and as we passed the fast-scale period, the time step is increased again and the simulation time takes shorter to compute. In particular, this indicates that the wall-clock time -- which is affected by factors including assembly time and the number of nonlinear and linear iterations -- is consistent with respect to time step size.

Finally, we find that the total energy is decreasing up to time $t = 713$, reaching $H_{D\!G} \approx 0.03717$. From then on, the total energy grows slowly to $0.03727$ at time $t=1000$, at which point its growth rate is smaller than $10^{-9}$. While this growth is unphysical and related to large time steps in the adaptive time discretization -- which has no theoretical guarantee for preserving the equation's energy dissipation property -- it is very small and only occurs towards the end of the simulation, where the remaining area of phase $\phi = 1$ slowly deforms into a vertical strip region and the timesteps taken are very large. Further, we note that mass is conserved up to an error of $10^{-12}$ (plot not shown here).

\textbf{Advection-dominated regime.} The last test case we consider includes a strong shear flow with low diffusivity. In such a scenario, the phase separation is directional with respect to the flow, forming filament structures \cite{o2007bubbles}. We consider a square domain $\Omega = [0, 1] \times [0, 1]$ that is periodic in the $x$-direction, with a flow field given by
\begin{equation}
\mathbf{u} = (1 + y) \mathbf{e}_x, \label{u_flow}
\end{equation}
for unit vector $\mathbf{e}_x$ in the $x$-coordinate direction. $\phi$ is set to be distributed uniformly at random within a strip $y \in [0.05, 0.95]$, and we consider homogeneous Neumann boundary conditions \eqref{Neumann_BCs} at $y=0$, $y=1$. Further, we set $d = 1/400$, $\epsilon = 1/50$.
In order to avoid alignment with the flow, we consider a pseudo-regular quadrilateral mesh as for test case \eqref{phi_flow}, with $64^2$ cells. The simulation is run up to $t = 20$, using the adaptive time stepper with initial time step $\delta t = 10^{-5}$. The resulting phase field is depicted in Figure \ref{CH_filaments}. As expected, we find that over time, the spinodal decomposition forms filaments along the direction of the flow.
\begin{figure}[ht]
\begin{center}
\includegraphics[width=1.\textwidth]{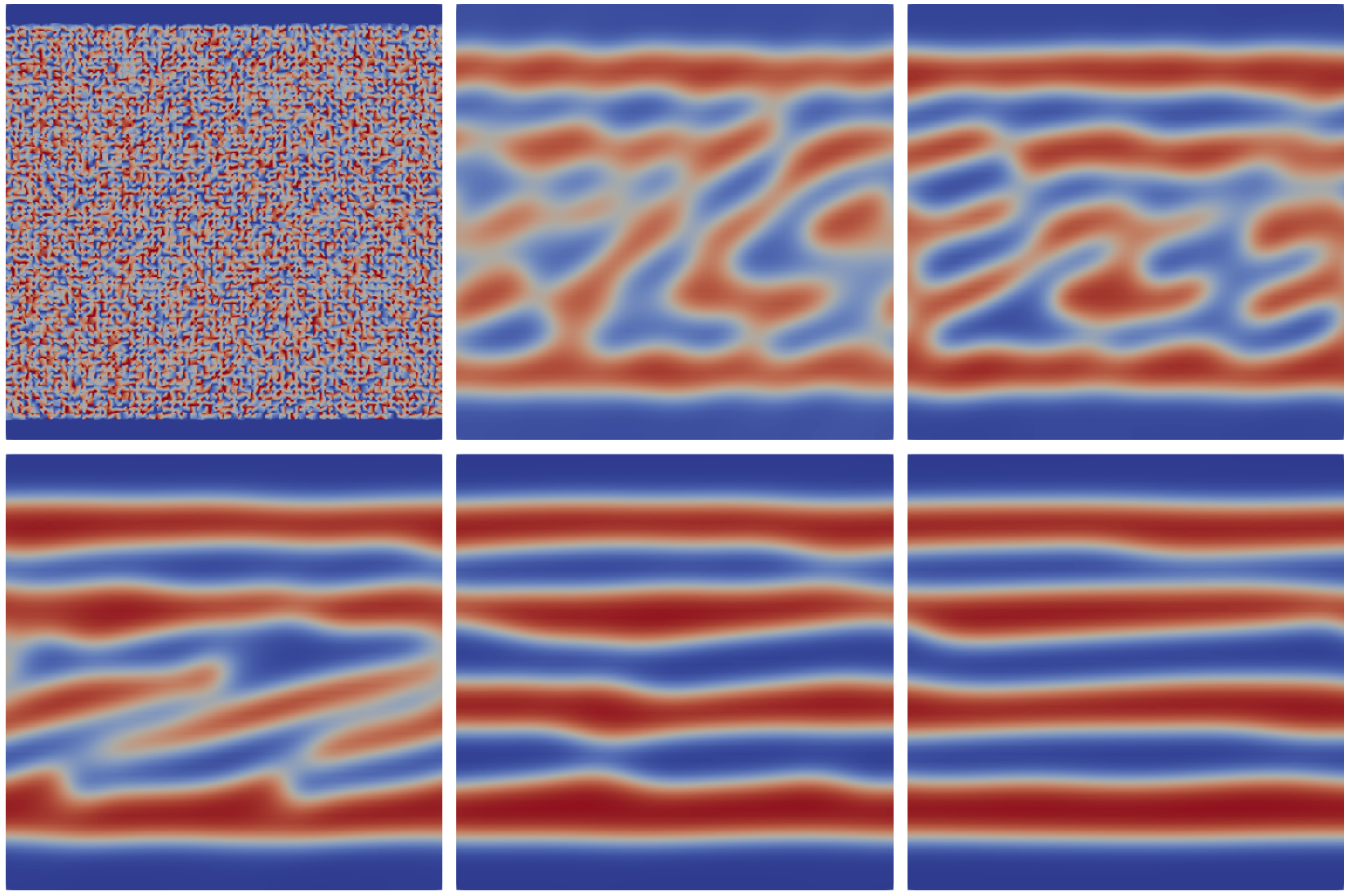}
\vspace{-5mm}
\caption{Phase field $\phi$ for advection dominant test case \eqref{u_flow} with flow along $x$-axis, at times $t \in \{0, 2, 3, 7, 10, 20\}$, from left to right, row-wise. Colors correspond to blue for $\phi = -1$, white for $\phi = 0$, and red for $\phi = 1$.} \label{CH_filaments}
\end{center}
\end{figure}
%
%
\section{Conclusion} \label{sec_conclusion}
In this work, we presented a novel spatial discretization for the Cahn-Hilliard equations with advection, consisting of a compatible finite element method with the phase field set in a discontinuous Galerkin space, and upwinding applied for the advection term. We demonstrated that the space-discretized equation is structure-preserving with respect to the mass conservation and energy dissipation. Further, we coupled the spatial discretization to an adaptive time stepping method, and derived a robust preconditioning strategy to solve for the resulting scheme's implicit stages.

In numerical tests we demonstrated the method's expected order of accuracy, stability, structure-preserving properties and solver robustness. Further, we showed the adaptive time stepping's capability to adjust both to small time steps in advance for fast dynamics such as mergers and absorptions of regions of equal phase, as well as large time steps for slow scale dynamics such as transition zone deformation for energy minimization. Together with our robust solver, this leads to an efficient, scheme for the Cahn-Hilliard equation with advection. Finally, we showed that, our discontinuous-Galerkin upwind-based discretization performs well in an advection-dominated regime and recovers the physically expected, filament shaped phase field pattern.

Possible future extensions to this work includes a study of coupling the advected Cahn-Hilliard scheme to the compressible or incompressible Navier-Stokes equations. Given our choice of space for the phase field, this could include standard discontinuous Galerkin based flux formulations of the latter equations, or mixed formulations with a divergence-conforming discretization for the flow field.

\textbf{Acknowledgments}\\
This work was supported by the Laboratory Directed Research and Development program of Los Alamos National Laboratory, under project number 20240261ER, as well as the U.S. Department of Energy Office of Fusion Energy Sciences Base Theory Program, at Los Alamos National Laboratory under contract No. 89233218CNA000001. The computations have been performed using resources of the National Energy Research Scientific Computing Center (NERSC), a U.S. Department of Energy Office of Science User Facility operated under Contract No. DE-AC02-05CH11231. Los Alamos National Laboratory report number LA-UR-24-31209.
%
%
%
%
\appendix
\section{Interior Penalty based discrete Laplacian} \label{app_IPDG}
Given an elliptic Neumann boundary condition problem for a field $u$ with $\mathbf{n} \cdot \nabla u |_{\partial \Omega} = g_u$, the interior penalty based discrete Laplacian is given by means of a variational form \cite{burman2005unified} as
\begin{align}
l(u, \psi) = \langle \nabla \psi, \nabla u \rangle \!-\!\! \int_\Gamma [\![\psi]\!]\{u\}dS \!-\!\! \int_\Gamma [\![u]\!]\{\psi\}dS \!+\!\! \int_\Gamma \frac{\kappa}{h_e} [\![u]\!][\![\psi]\!]dS \!-\!\! \int_{\partial \Omega} \psi g_u \; dS && \forall \psi \in \mathbb{V}_{DG}, \label{IPDG}
\end{align}
where $\Gamma$ denotes the set of interior facets, and $h_e = \tfrac{|K^+| + |K^-|}{2|\partial K|}$ is a length scale for the given facet. $|K^{\pm}|$ denotes the facet's adjacent cell's volumes, and $|\partial K|$ denotes the facet area. Further,
\begin{equation}
[\![\psi]\!] \coloneqq \left(\psi^+ - \psi^-\right), \hspace{1cm} \{\psi\} \coloneqq \frac{1}{2} \left(\psi^+ + \psi^-\right).
\end{equation}
For the preconditioner \eqref{DG_diag_precon}, we drop the Neumann boundary condition term, and futher the penalty parameter is set to $\kappa = 10$ in the numerical results section \ref{sec_Numerical_results}.
%
%
%
\bibliographystyle{plain}
\bibliography{paper_Cahn_Hilliard}

\begin{thebibliography}{10}

\bibitem{akrivis2019energy}
G.~Akrivis, B.~Li, and D.~Li.
\newblock Energy-decaying extrapolated {RK}--{SAV} methods for the
  {A}llen--{C}ahn and {C}ahn--{H}illiard equations.
\newblock {\em SIAM Journal on Scientific Computing}, 41(6):A3703--A3727, 2019.

\bibitem{anderson1998diffuse}
D.~M. Anderson, G.~B. McFadden, and A.~A. Wheeler.
\newblock Diffuse-interface methods in fluid mechanics.
\newblock {\em Annual review of fluid mechanics}, 30(1):139--165, 1998.

\bibitem{anderson2021mfem}
R.~Anderson, J.~Andrej, A.~Barker, J.~Bramwell, J.-S. Camier, J.~Cerveny,
  V.~Dobrev, Y.~Dudouit, A.~Fisher, T.~Kolev, et~al.
\newblock {MFEM}: A modular finite element methods library.
\newblock {\em Computers \& Mathematics with Applications}, 81:42--74, 2021.

\bibitem{arnold2006finite}
D.~N. Arnold, R.~S. Falk, and R.~Winther.
\newblock Finite element exterior calculus, homological techniques, and
  applications.
\newblock {\em Acta Numerica}, 15:1--155, 2006.

\bibitem{balay2019petsc}
S.~Balay, S.~Abhyankar, M.~Adams, J.~Brown, P.~Brune, K.~Buschelman, L.~Dalcin,
  A.~Dener, V.~Eijkhout, W.~Gropp, et~al.
\newblock {PETS}c users manual.
\newblock 2019.

\bibitem{bank1985transient}
R.~E. Bank, W.~M. Coughran, W.~Fichtner, E.~H. Grosse, D.~J. Rose, and R.~K.
  Smith.
\newblock Transient simulation of silicon devices and circuits.
\newblock {\em IEEE Transactions on Computer-Aided Design of Integrated
  Circuits and Systems}, 4(4):436--451, 1985.

\bibitem{benzi2005numerical}
M.~Benzi, G.~H. Golub, and J.~Liesen.
\newblock Numerical solution of saddle point problems.
\newblock {\em Acta numerica}, 14:1--137, 2005.

\bibitem{boffi2013mixed}
D.~Boffi, F.~Brezzi, M.~Fortin, et~al.
\newblock {\em Mixed finite element methods and applications}, volume~44.
\newblock Springer, 2013.

\bibitem{brenner2008mathematical}
S.~C. Brenner.
\newblock {\em The mathematical theory of finite element methods}.
\newblock Springer, 2008.

\bibitem{brenner2018robust}
S.~C. Brenner, A.~E. Diegel, and L.-Y. Sung.
\newblock A robust solver for a mixed finite element method for the
  {C}ahn--{H}illiard equation.
\newblock {\em Journal of Scientific Computing}, 77:1234--1249, 2018.

\bibitem{burman2005unified}
E.~Burman.
\newblock A unified analysis for conforming and nonconforming stabilized finite
  element methods using interior penalty.
\newblock {\em SIAM journal on numerical analysis}, 43(5):2012--2033, 2005.

\bibitem{cahn1958free}
J.~W. Cahn and J.~E. Hilliard.
\newblock Free energy of a nonuniform system. {I}. interfacial free energy.
\newblock {\em The Journal of chemical physics}, 28(2):258--267, 1958.

\bibitem{chen2002phase}
L.-Q. Chen.
\newblock Phase-field models for microstructure evolution.
\newblock {\em Annual review of materials research}, 32(1):113--140, 2002.

\bibitem{chen2019positivity}
W.~Chen, C.~Wang, X.~Wang, and S.~M. Wise.
\newblock Positivity-preserving, energy stable numerical schemes for the
  {C}ahn-{H}illiard equation with logarithmic potential.
\newblock {\em Journal of Computational Physics: X}, 3:100031, 2019.

\bibitem{dehghan2021numerical}
M.~Dehghan and Z.~Gharibi.
\newblock Numerical analysis of fully discrete energy stable weak {G}alerkin
  finite element scheme for a coupled {C}ahn-{H}illiard-{N}avier-{S}tokes
  phase-field model.
\newblock {\em Applied Mathematics and Computation}, 410:126487, 2021.

\bibitem{diegel2015analysis}
A.~E. Diegel, X.~H. Feng, and S.~M. Wise.
\newblock Analysis of a mixed finite element method for a
  {C}ahn--{H}illiard--{D}arcy--{S}tokes system.
\newblock {\em SIAM Journal on Numerical Analysis}, 53(1):127--152, 2015.

\bibitem{diegel2016stability}
A.~E. Diegel, C.~Wang, and S.~M. Wise.
\newblock Stability and convergence of a second-order mixed finite element
  method for the {C}ahn--{H}illiard equation.
\newblock {\em IMA Journal of Numerical Analysis}, 36(4):1867--1897, 2016.

\bibitem{duran2008mixed}
R.~Dur{\'a}n.
\newblock Mixed finite element methods.
\newblock {\em Mixed finite elements, compatibility conditions, and
  applications}, pages 1--44, 2008.

\bibitem{eldred2019quasi}
C.~Eldred, T.~Dubos, and E.~Kritsikis.
\newblock A quasi-{H}amiltonian discretization of the thermal shallow water
  equations.
\newblock {\em Journal of Computational Physics}, 379:1--31, 2019.

\bibitem{falgout2002hypre}
R.~D. Falgout and U.~M. Yang.
\newblock hypre: A library of high performance preconditioners.
\newblock In {\em International Conference on computational science}, pages
  632--641. Springer, 2002.

\bibitem{feng2016analysis}
X.~Feng, Y.~Li, and Y.~Xing.
\newblock Analysis of mixed interior penalty discontinuous {G}alerkin methods
  for the {C}ahn--{H}illiard equation and the {H}ele--{S}haw flow.
\newblock {\em SIAM Journal on Numerical Analysis}, 54(2):825--847, 2016.

\bibitem{fu2020divergence}
G.~Fu.
\newblock A divergence-free {HDG} scheme for the {C}ahn-{H}illiard phase-field
  model for two-phase incompressible flow.
\newblock {\em Journal of Computational Physics}, 419:109671, 2020.

\bibitem{fu2021linear}
G.~Fu and D.~Han.
\newblock A linear second-order in time unconditionally energy stable finite
  element scheme for a {C}ahn--{H}illiard phase-field model for two-phase
  incompressible flow of variable densities.
\newblock {\em Computer Methods in Applied Mechanics and Engineering},
  387:114186, 2021.

\bibitem{fu2024higher}
Zhaohui Fu, Jie Shen, and Jiang Yang.
\newblock Higher-order energy-decreasing exponential time differencing
  runge-kutta methods for gradient flows.
\newblock {\em arXiv preprint arXiv:2402.15142}, 2024.

\bibitem{fu2024energy}
Zhaohui Fu, Tao Tang, and Jiang Yang.
\newblock Energy diminishing implicit-explicit runge--kutta methods for
  gradient flows.
\newblock {\em Mathematics of Computation}, 2024.

\bibitem{fu2022energy}
Zhaohui Fu and Jiang Yang.
\newblock Energy-decreasing exponential time differencing runge--kutta methods
  for phase-field models.
\newblock {\em Journal of Computational Physics}, 454:110943, 2022.

\bibitem{FiredrakeUserManual}
D.~A. Ham, P.~H.~J. Kelly, L.~Mitchell, C.~J. Cotter, R.~C. Kirby, K.~Sagiyama,
  N.~Bouziani, S.~Vorderwuelbecke, T.~J. Gregory, J.~Betteridge, D.~R. Shapero,
  R.~W. Nixon-Hill, C.~J. Ward, P.~E. Farrell, P.~D. Brubeck, I.~Marsden, T.~H.
  Gibson, M.~Homolya, T.~Sun, A.~T.~T. McRae, F.~Luporini, A.~Gregory,
  M.~Lange, S.~W. Funke, F.~Rathgeber, G.-T. Bercea, and G.~R. Markall.
\newblock {\em Firedrake User Manual}.
\newblock Imperial College London and University of Oxford and Baylor
  University and University of Washington, first edition edition, 5 2023.

\bibitem{hosea1996analysis}
M.~E. Hosea and L.~F. Shampine.
\newblock Analysis and implementation of {TR}-{BDF2}.
\newblock {\em Applied Numerical Mathematics}, 20(1-2):21--37, 1996.

\bibitem{hughes2010stabilized}
T.~J.~R. Hughes, G.~Scovazzi, and T.~E. Tezduyar.
\newblock Stabilized methods for compressible flows.
\newblock {\em Journal of Scientific Computing}, 43:343--368, 2010.

\bibitem{kay2009discontinuous}
D.~Kay, V.~Styles, and E.~S{\"u}li.
\newblock Discontinuous {G}alerkin finite element approximation of the
  {C}ahn--{H}illiard equation with convection.
\newblock {\em SIAM Journal on Numerical Analysis}, 47(4):2660--2685, 2009.

\bibitem{khanwale2022fully}
M.~A. Khanwale, K.~Saurabh, M.~Fernando, V.~M. Calo, H.~Sundar, J.~A.
  Rossmanith, and B.~Ganapathysubramanian.
\newblock A fully-coupled framework for solving {C}ahn-{H}illiard
  {N}avier-{S}tokes equations: Second-order, energy-stable numerical methods on
  adaptive octree based meshes.
\newblock {\em Computer Physics Communications}, 280:108501, 2022.

\bibitem{kim2004conservative}
J.~Kim, K.~Kang, and J.~Lowengrub.
\newblock Conservative multigrid methods for {C}ahn--{H}illiard fluids.
\newblock {\em Journal of Computational Physics}, 193(2):511--543, 2004.

\bibitem{kim2016basic}
J.~Kim, S.~Lee, Y.~Choi, S.-M. Lee, and D.~Jeong.
\newblock Basic principles and practical applications of the {C}ahn--{H}illiard
  equation.
\newblock {\em Mathematical Problems in Engineering}, 2016(1):9532608, 2016.

\bibitem{kirk2023numerical}
K.~L.~A. Kirk, B.~Riviere, and R.~Masri.
\newblock Numerical analysis of a hybridized discontinuous {G}alerkin method
  for the {C}ahn--{H}illiard problem.
\newblock {\em IMA Journal of Numerical Analysis}, page drad075, 2023.

\bibitem{kuzmin2010guide}
D.~Kuzmin.
\newblock A guide to numerical methods for transport equations, 2010.

\bibitem{lee2014physical}
D.~Lee, J.-Y. Huh, D.~Jeong, J.~Shin, A.~Yun, and J.~Kim.
\newblock Physical, mathematical, and numerical derivations of the
  {C}ahn--{H}illiard equation.
\newblock {\em Computational Materials Science}, 81:216--225, 2014.

\bibitem{liu2019numerical}
C.~Liu, F.~Frank, and B.~Rivi{\`e}re.
\newblock Numerical error analysis for nonsymmetric interior penalty
  discontinuous galerkin method of {C}ahn--{H}illiard equation.
\newblock {\em Numerical Methods for Partial Differential Equations},
  35(4):1509--1537, 2019.

\bibitem{liu2021unconditionally}
H.~Liu and P.~Yin.
\newblock Unconditionally energy stable discontinuous {G}alerkin schemes for
  the {C}ahn--{H}illiard equation.
\newblock {\em Journal of Computational and Applied Mathematics}, 390:113375,
  2021.

\bibitem{manteuffel2019nonsymmetric}
T.~A. Manteuffel, S.~M{\"u}nzenmaier, J.~Ruge, and B.~S. Southworth.
\newblock Nonsymmetric reduction-based algebraic multigrid.
\newblock {\em SIAM Journal on Scientific Computing}, 41(5):S242--S268, 2019.

\bibitem{manteuffel2018nonsymmetric}
T.~A. Manteuffel, J.~Ruge, and B.~S. Southworth.
\newblock Nonsymmetric algebraic multigrid based on local approximate ideal
  restriction (l{AIR}).
\newblock {\em SIAM Journal on Scientific Computing}, 40(6):A4105--A4130, 2018.

\bibitem{miranville2019cahn}
A.~Miranville.
\newblock {\em The {C}ahn--{H}illiard equation: recent advances and
  applications}.
\newblock SIAM, 2019.

\bibitem{o2007bubbles}
L.~{\'O}~N{\'a}raigh and J.-L. Thiffeault.
\newblock Bubbles and filaments: Stirring a {C}ahn--{H}illiard fluid.
\newblock {\em Physical Review E---Statistical, Nonlinear, and Soft Matter
  Physics}, 75(1):016216, 2007.

\bibitem{ranocha2022optimized}
H.~Ranocha, L.~Dalcin, M.~Parsani, and D.~I. Ketcheson.
\newblock Optimized {R}unge-{K}utta methods with automatic step size control
  for compressible computational fluid dynamics.
\newblock {\em Communications on Applied Mathematics and Computation},
  4(4):1191--1228, 2022.

\bibitem{roache2002code}
P.~J. Roache.
\newblock Code verification by the method of manufactured solutions.
\newblock {\em J. Fluids Eng.}, 124(1):4--10, 2002.

\bibitem{ruge1987algebraic}
J.~W. Ruge and K.~St{\"u}ben.
\newblock Algebraic multigrid.
\newblock In {\em Multigrid methods}, pages 73--130. SIAM, 1987.

\bibitem{shen2018scalar}
Jie Shen, Jie Xu, and Jiang Yang.
\newblock The scalar auxiliary variable (sav) approach for gradient flows.
\newblock {\em Journal of Computational Physics}, 353:407--416, 2018.

\bibitem{shen2019new}
Jie Shen, Jie Xu, and Jiang Yang.
\newblock A new class of efficient and robust energy stable schemes for
  gradient flows.
\newblock {\em SIAM Review}, 61(3):474--506, 2019.

\bibitem{soderlind2002automatic}
G.~S{\"o}derlind.
\newblock Automatic control and adaptive time-stepping.
\newblock {\em Numerical Algorithms}, 31:281--310, 2002.

\bibitem{soderlind2006adaptive}
G.~S{\"o}derlind and L.~Wang.
\newblock Adaptive time-stepping and computational stability.
\newblock {\em Journal of Computational and Applied Mathematics},
  185(2):225--243, 2006.

\bibitem{wimmer2020energy}
G.~A. Wimmer, C.~J. Cotter, and W.~Bauer.
\newblock Energy conserving upwinded compatible finite element schemes for the
  rotating shallow water equations.
\newblock {\em Journal of Computational Physics}, 401:109016, 2020.

\bibitem{wu2014stabilized}
X.~Wu, G.~J. Van~Zwieten, and K.~G. van~der Zee.
\newblock Stabilized second-order convex splitting schemes for
  {C}ahn--{H}illiard models with application to diffuse-interface tumor-growth
  models.
\newblock {\em International journal for numerical methods in biomedical
  engineering}, 30(2):180--203, 2014.

\bibitem{zhao2023numerical}
W.~Zhao and Q.~Guan.
\newblock Numerical analysis of energy stable weak {G}alerkin schemes for the
  {C}ahn--{H}illiard equation.
\newblock {\em Communications in Nonlinear Science and Numerical Simulation},
  118:106999, 2023.

\end{thebibliography}
\end{document}